\documentclass[12pt,reqno]{amsart}%
\usepackage{amsmath}
\usepackage{amsfonts}
\usepackage{amssymb}
\usepackage{amsthm}
\usepackage{amsbsy}
\usepackage{amscd}
\usepackage{fancyhdr}
\usepackage[usenames,dvipsnames,svgnames,x11names,hyperref]{xcolor}
\usepackage{geometry}
\usepackage{graphicx}
\usepackage{hyperref}
\providecommand{\U}[1]{\protect\rule{.1in}{.1in}}
\hypersetup{backref=true,
pagebackref=true,
hyperindex=true,
colorlinks=true,
breaklinks=true,
urlcolor=NavyBlue,
linkcolor=Fuchsia,
bookmarks=true,
bookmarksopen=false,
filecolor=black,
citecolor=ForestGreen,
linkbordercolor=red
}
\newtheorem{theorem}{Theorem}[section]
\theoremstyle{plain}

\newtheorem{corollary}{Corollary}[section]

\newtheorem{lemma}{Lemma}[section]

\numberwithin{equation}{section}
\allowdisplaybreaks
\AtBeginDocument{\hypersetup{pdfborder={0 0 0.01}}}
\begin{document}
\title[$L^{p}$-Caffarelli-Kohn-Nirenberg]{$L^{p}$-Caffarelli-Kohn-Nirenberg inequalities and their stabilities}
\author{Anh Xuan Do}
\address{Anh Xuan Do: Department of Mathematics\\
University of Connecticut\\
Storrs, CT 06269, USA}
\email{anh.do@uconn.edu}
\author{Joshua Flynn}
\address{Joshua Flynn: CRM/ISM and McGill University\\
Montr\'{e}al, QC H3A0G4, Canada}
\email{joshua.flynn@mcgill.ca}
\author{Nguyen Lam}
\address{Nguyen Lam: School of Science and the Environment\\
Grenfell Campus, Memorial University of Newfoundland\\
Corner Brook, NL A2H5G4, Canada }
\email{nlam@grenfell.mun.ca}
\author{Guozhen Lu }
\address{Guozhen Lu: Department of Mathematics\\
University of Connecticut\\
Storrs, CT 06269, USA}
\email{guozhen.lu@uconn.edu}
\thanks{A. Do and G. Lu were partly supported by collaboration grants and Simons
Fellowship from the Simons foundation. N. Lam was partially supported by an
NSERC Discovery Grant. }
\date{\today}

\begin{abstract}
We establish a general identity (Theorem \ref{T1}) that implies both the
$L^{p}$-Hardy identities and the $L^{p}$-Caffarelli-Kohn-Nirenberg identities
(Theorems \ref{T2} and \ref{T3}) and $L^{p}$-Hardy inequalities and the
$L^{p}$-Caffarelli-Kohn-Nirenberg inequalities (Theorems \ref{T4}, \ref{T5})).
Weighted $L^{p}$-Caffarelli-Kohn-Nirenberg inequalities with nonradial weights
are also obtained. (Theorem \ref{T5.1}). Our results provide simple
interpretations to the sharp constants, as well as the existence and
non-existence of the optimizers, of several $L^{p}$-Hardy and $L^{p}%
$-Caffarelli-Kohn-Nirenberg inequalities. As applications of our main results,
we are able to establish stabilities of a class of $L^{2}$ and $L^{p}%
$-Caffarelli-Kohn-Nirenberg inequalities. (Theorems \ref{T6} and \ref{T8}.) We
also derive the best constants and explicit extremal functions for a large
family of $L^{2}$ and $L^{p}$ Caffarelli-Kohn-Nirenberg inequalities.
(Corollaries \ref{2CKN} and \ref{CKN}.)

\end{abstract}
\maketitle

\section{Introduction}

The $L^{p}$-Hardy type inequality of the form%
\begin{equation}
\int_{\Omega}A\left(  x\right)  \left\vert \nabla u\right\vert ^{p}dx\geq
\int_{\Omega}B\left(  x\right)  \left\vert u\right\vert ^{p}dx \label{pHardy}%
\end{equation}
is one of the most important inequalities in modern mathematics. It plays an
important role in partial differential equations, mathematical physics,
differential geometry, spectral analysis, etc, and has been widely studied in
the literature. We refer the interested reader to the celebrated paper
\cite{BV97} for some pioneering improvements, and to the monographs \cite{BEL,
GM1, KMP2007, KP, Maz11, OK}, for instance, for many detailed developments and applications.

Many works have been devoted to study the conditions of the potential pair
$\left(  A,B\right)  $ such that the $L^{p}$-Hardy type inequality
(\ref{pHardy}) holds for all $u\in C_{0}^{\infty}\left(  \Omega\right)  $. For
instance, Frank and Seiringer provided in \cite{FS08} a general method in
terms of nonlinear ground state representations to derive the sharp local and
nonlocal Hardy inequalities. In the setting of $L^{2}$-spaces, Ghoussoub and
Moradifam proposed the notion of Bessel pair in \cite{GM1}, and used it to
study many improvements of the $L^{2}$-Hardy type inequality with radial
weights. This notion of Bessel pair has also been applied to investigate
further the $L^{2}$-Hardy type identities and inequalities in \cite{FLL21,
LLZ19, LLZ20, Wang}, to name just a few. See also \cite{HY21} for a more
general setting. In \cite{DLL22}, the authors introduced the notion of
$p$-Bessel pair and used it to establish several $L^{p}$-Hardy type identities
and inequalities.

In \cite{CFLL23}, the authors proved the following general identity

\begin{theorem}
\label{T0}\textit{Let }$0<R\leq\infty$\textit{, }$A$\textit{ and }$B$\textit{
be }$C^{1}$-\textit{functions on }$\left(  0,R\right)  $ and let
\[
C\left(  r\right)  =\left(  A\left(  r\right)  B\left(  r\right)  \right)
^{\prime}+\left(  N-1\right)  \frac{A\left(  r\right)  B\left(  r\right)  }%
{r}-B^{2}\left(  r\right)  \text{.}%
\]
Then for all $\alpha\in\mathbb{R}\setminus\left\{  0\right\}  $ and $u\in
C_{0}^{\infty}\left(  B_{R}\setminus\left\{  0\right\}  \right)  $, we have%
\begin{align*}
&  \left\vert \alpha\right\vert ^{2}{\int\limits_{B_{R}}}A^{2}\left(
\left\vert x\right\vert \right)  \left\vert \frac{x}{\left\vert x\right\vert
}\cdot\nabla u\left(  x\right)  \right\vert ^{2}\mathrm{dx}+\frac
{1}{\left\vert \alpha\right\vert ^{2}}{\int\limits_{B_{R}}}B^{2}\left(
\left\vert x\right\vert \right)  \left\vert u\left(  x\right)  \right\vert
^{2}\mathrm{dx}\\
&  ={\int\limits_{B_{R}}}\left[  C\left(  \left\vert x\right\vert \right)
+B^{2}\left(  \left\vert x\right\vert \right)  \right]  \left\vert
u\right\vert ^{2}\mathrm{dx}+{\int\limits_{B_{R}}}\left\vert \alpha A\left(
\left\vert x\right\vert \right)  \frac{x}{\left\vert x\right\vert }\cdot\nabla
u+\frac{1}{\alpha}B\left(  \left\vert x\right\vert \right)  u\right\vert
^{2}\mathrm{dx}%
\end{align*}
and%
\begin{align*}
&  \left\vert \alpha\right\vert ^{2}{\int\limits_{B_{R}}}A^{2}\left(
\left\vert x\right\vert \right)  \left\vert \nabla u\right\vert ^{2}%
\mathrm{dx}+\frac{1}{\left\vert \alpha\right\vert ^{2}}{\int\limits_{B_{R}}%
}B^{2}\left(  \left\vert x\right\vert \right)  \left\vert u\right\vert
^{2}\mathrm{dx}\\
&  ={\int\limits_{B_{R}}}\left[  C\left(  \left\vert x\right\vert \right)
+B^{2}\left(  \left\vert x\right\vert \right)  \right]  \left\vert u\left(
x\right)  \right\vert ^{2}\mathrm{dx}+{\int\limits_{B_{R}}}\left\vert \alpha
A\left(  \left\vert x\right\vert \right)  \nabla u\left(  x\right)  +\frac
{1}{\alpha}B\left(  \left\vert x\right\vert \right)  u\left(  x\right)
\frac{x}{\left\vert x\right\vert }\right\vert ^{2}\mathrm{dx}.
\end{align*}

\end{theorem}

When $\left\vert \alpha\right\vert =1$, the above Theorem gives a general
$L^{2}$-Hardy type inequality with radial weights that unifies and improves
several known $L^{2}$-Hardy type inequality in the literature. Futhermore,
when optimizing $\alpha$, Theorem 1.1 yields the $L^{2}$%
-Caffarelli-Kohn-Nirenberg (CKN) type inequality. Therefore, $L^{2}$-Hardy
inequalities can be considered as the non-optimal (scale non-invariant)
$L^{2}$-CKN inequalities. Also, Theorem \ref{T0} can be used to derive several
$L^{2}$-Hardy inequalities and the $L^{2}$-CKN inequalities with radial
weights. Moreover, the identity forms can be used to explain for the
attainability/unattainability of the sharp constants and the existence of
optimizers/virtual optimizers of the $L^{2}$-Hardy inequalities and the
$L^{2}$-CKN inequalities.

The first principal goal of this paper is to extend the above result to the
$L^{p}$ setting with general weights. In particular, we will set up some
identities that implies the $L^{p}$-Hardy identities and inequalities and the
$L^{p}$-CKN identities and inequalities. Moreover, we will study these
identities for potentials that are not radial in general. More precisely, let
$p>1$, $\overrightarrow{a},\overrightarrow{b}$ be vectors in $\mathbb{R}^{n}$,
$n\geq1$, and let
\[
\mathcal{R}_{p}\left(  \overrightarrow{a},\overrightarrow{b}\right)
=\left\vert \overrightarrow{b}\right\vert ^{p}+\left(  p-1\right)  \left\vert
\overrightarrow{a}\right\vert ^{p}-p\left\vert \overrightarrow{a}\right\vert
^{p-2}\overrightarrow{a}\cdot\overrightarrow{b}.
\]
Then, our first main result of this paper is following identities

\begin{theorem}
\label{T1}Let $\Omega$ be an open set in $\mathbb{R}^{N}$, $N\geq1$, $p>1$,
$\alpha>0$, $A\in C^{1}\left(  \Omega\right)  $ and $\overrightarrow{X}\in
C^{1}\left(  \Omega,\mathbb{R}^{N}\right)  $. Then for any $u\in C_{0}%
^{1}\left(  \Omega\right)  $, we have%
\begin{align*}
&  \alpha^{p}\int_{\Omega}A\left\vert \nabla u\right\vert ^{p}dx+\frac{\left(
p-1\right)  }{\alpha^{\frac{p}{p-1}}}\int_{\Omega}A\left\vert \overrightarrow
{X}\right\vert ^{p}\left\vert u\right\vert ^{p}dx\\
&  =-\int_{\Omega}\operatorname{div}\left(  A\left\vert \overrightarrow
{X}\right\vert ^{p-2}\overrightarrow{X}\right)  \left\vert u\right\vert
^{p}dx+\int_{\Omega}A\mathcal{R}_{p}\left(  \frac{1}{\alpha^{\frac{1}{p-1}}%
}u\overrightarrow{X},\alpha\nabla u\right)  dx
\end{align*}
and
\begin{align*}
&  \alpha^{p}\int_{\Omega}A\left\vert \frac{\overrightarrow{X}}{\left\vert
\overrightarrow{X}\right\vert }\cdot\nabla u\right\vert ^{p}dx+\frac{\left(
p-1\right)  }{\alpha^{\frac{p}{p-1}}}\int_{\Omega}A\left\vert \overrightarrow
{X}\right\vert ^{p}\left\vert u\right\vert ^{p}dx\\
&  =-\int_{\Omega}\operatorname{div}\left(  A\left\vert \overrightarrow
{X}\right\vert ^{p-2}\overrightarrow{X}\right)  \left\vert u\right\vert
^{p}dx+\int_{\Omega}A\mathcal{R}_{p}\left(  \frac{1}{\alpha^{\frac{1}{p-1}}%
}u\left\vert \overrightarrow{X}\right\vert ,\alpha\frac{\overrightarrow{X}%
}{\left\vert \overrightarrow{X}\right\vert }\cdot\nabla u\right)  dx
\end{align*}

\end{theorem}

By choosing $\alpha=1$, we obtain the following $L^{p}$-Hardy type identities

\begin{theorem}
\label{T2}Let $\Omega$ be an open set in $\mathbb{R}^{N}$, $N\geq1$, $p>1$,
$A\in C^{1}\left(  \Omega\right)  $ and $\overrightarrow{X}\in C^{1}\left(
\Omega,\mathbb{R}^{N}\right)  $. Then for any $u\in C_{0}^{1}\left(
\Omega\right)  $, we have%
\begin{align*}
&  \int_{\Omega}A\left\vert \nabla u\right\vert ^{p}dx-\int_{\Omega}\left(
-\operatorname{div}\left(  A\left\vert \overrightarrow{X}\right\vert
^{p-2}\overrightarrow{X}\right)  -\left(  p-1\right)  A\left\vert
\overrightarrow{X}\right\vert ^{p}\right)  \left\vert u\right\vert ^{p}dx\\
&  =\int_{\Omega}A\mathcal{R}_{p}\left(  u\overrightarrow{X},\nabla u\right)
dx
\end{align*}
and%
\begin{align*}
&  \int_{\Omega}A\left\vert \frac{\overrightarrow{X}}{\left\vert
\overrightarrow{X}\right\vert }\cdot\nabla u\right\vert ^{p}dx-\int_{\Omega
}\left(  -\operatorname{div}\left(  A\left\vert \overrightarrow{X}\right\vert
^{p-2}\overrightarrow{X}\right)  -\left(  p-1\right)  A\left\vert
\overrightarrow{X}\right\vert ^{p}\right)  \left\vert u\right\vert ^{p}dx\\
&  =\int_{\Omega}A\mathcal{R}_{p}\left(  u\left\vert \overrightarrow
{X}\right\vert ,\frac{\overrightarrow{X}}{\left\vert \overrightarrow
{X}\right\vert }\cdot\nabla u\right)  dx.
\end{align*}

\end{theorem}

On the other hand, by optimizing $\alpha$ (i.e. by choosing $\alpha=\left(
\frac{\int_{\Omega}A\left\vert \overrightarrow{X}\right\vert ^{p}\left\vert
u\right\vert ^{p}dx}{\int_{\Omega}A\left\vert \nabla u\right\vert ^{p}%
dx}\right)  ^{\frac{p-1}{p^{2}}}$ and $\left(  \frac{\int_{\Omega}A\left\vert
\overrightarrow{X}\right\vert ^{p}\left\vert u\right\vert ^{p}dx}{\int
_{\Omega}A\left\vert \frac{\overrightarrow{X}}{\left\vert \overrightarrow
{X}\right\vert }\cdot\nabla u\right\vert ^{p}dx}\right)  ^{\frac{p-1}{p^{2}}}$
respectively), we obtain the $L^{p}$-CKN identities

\begin{theorem}
\label{T3}Let $\Omega$ be an open set in $\mathbb{R}^{N}$, $N\geq1$, $p>1$,
$A\in C^{1}\left(  \Omega\right)  ,$ $A\geq0$, and $\overrightarrow{X}\in
C^{1}\left(  \Omega,\mathbb{R}^{N}\right)  $. Then for any $u\in C_{0}%
^{1}\left(  \Omega\right)  $, we have%
\begin{align*}
&  \left(  \int_{\Omega}A\left\vert \nabla u\right\vert ^{p}dx\right)
^{\frac{1}{p}}\left(  \int_{\Omega}A\left\vert \overrightarrow{X}\right\vert
^{p}\left\vert u\right\vert ^{p}dx\right)  ^{\frac{p-1}{p}}+\frac{1}{p}%
\int_{\Omega}\operatorname{div}\left(  A\left\vert \overrightarrow
{X}\right\vert ^{p-2}\overrightarrow{X}\right)  \left\vert u\right\vert
^{p}dx\\
&  =\frac{1}{p}\int_{\Omega}A\mathcal{R}_{p}\left(  \left(  \frac{\int
_{\Omega}A\left\vert \nabla u\right\vert ^{p}dx}{\int_{\Omega}A\left\vert
\overrightarrow{X}\right\vert ^{p}\left\vert u\right\vert ^{p}dx}\right)
^{\frac{1}{p^{2}}}u\overrightarrow{X},\left(  \frac{\int_{\Omega}A\left\vert
\overrightarrow{X}\right\vert ^{p}\left\vert u\right\vert ^{p}dx}{\int
_{\Omega}A\left\vert \nabla u\right\vert ^{p}dx}\right)  ^{\frac{p-1}{p^{2}}%
}\nabla u\right)  dx,
\end{align*}
and%
\begin{align*}
&  \left(  \int_{\Omega}A\left\vert \frac{\overrightarrow{X}}{\left\vert
\overrightarrow{X}\right\vert }\cdot\nabla u\right\vert ^{p}dx\right)
^{\frac{1}{p}}\left(  \int_{\Omega}A\left\vert \overrightarrow{X}\right\vert
^{p}\left\vert u\right\vert ^{p}dx\right)  ^{\frac{p-1}{p}}+\frac{1}{p}%
\int_{\Omega}\operatorname{div}\left(  A\left\vert \overrightarrow
{X}\right\vert ^{p-2}\overrightarrow{X}\right)  \left\vert u\right\vert
^{p}dx\\
&  =\frac{1}{p}\int_{\Omega}A\mathcal{R}_{p}\left(  \left(  \frac{\int
_{\Omega}A\left\vert \frac{\overrightarrow{X}}{\left\vert \overrightarrow
{X}\right\vert }\cdot\nabla u\right\vert ^{p}dx}{\int_{\Omega}A\left\vert
\overrightarrow{X}\right\vert ^{p}\left\vert u\right\vert ^{p}dx}\right)
^{\frac{1}{p^{2}}}u\left\vert \overrightarrow{X}\right\vert ,\left(
\frac{\int_{\Omega}A\left\vert \overrightarrow{X}\right\vert ^{p}\left\vert
u\right\vert ^{p}dx}{\int_{\Omega}A\left\vert \frac{\overrightarrow{X}%
}{\left\vert \overrightarrow{X}\right\vert }\cdot\nabla u\right\vert ^{p}%
dx}\right)  ^{\frac{p-1}{p^{2}}}\frac{\overrightarrow{X}}{\left\vert
\overrightarrow{X}\right\vert }\cdot\nabla u\right)  dx.
\end{align*}

\end{theorem}

Now, to derive the $L^{p}$-Hardy inequalities and $L^{p}$-CKN inequalities, we
state the following elementary estimates of $\mathcal{R}_{p}$ (see, for
instance, \cite{CKLL, DLL22}):

\begin{lemma}
\label{L1}Let $p>1$ and $n\geq1$. Then

\begin{enumerate}
\item $\mathcal{R}_{p}\left(  \overrightarrow{a},\overrightarrow{b}\right)
\geq0$ for all $\overrightarrow{a},\overrightarrow{b}\in\mathbb{R}^{n}$.
Moreover, $\mathcal{R}_{p}\left(  \overrightarrow{a},\overrightarrow
{b}\right)  =0$ if and only if $\overrightarrow{a}=\overrightarrow{b}$.

\item Let $p\geq2$. Then there exists $M_{p}\in\left(  0,1\right]  $ such that
$\mathcal{R}_{p}\left(  \overrightarrow{a},\overrightarrow{b}\right)  \geq
M_{p}\left\vert \overrightarrow{b}-\overrightarrow{a}\right\vert ^{p}$ for all
$\overrightarrow{a},\overrightarrow{b}\in\mathbb{R}^{n}$.
\end{enumerate}
\end{lemma}

As consequences of Theorem \ref{T2}, Theorem \ref{T3} and Lemma \ref{L1}, we
obtain the following $L^{p}$-Hardy inequalities and $L^{p}$-CKN inequalities:

\begin{theorem}
\label{T4}Let $\Omega$ be an open set in $\mathbb{R}^{N}$, $N\geq1$, $p>1$,
$A\in C^{1}\left(  \Omega\right)  ,$ $A\geq0$, and $\overrightarrow{X}\in
C^{1}\left(  \Omega,\mathbb{R}^{N}\right)  $. Then for any $u\in C_{0}%
^{1}\left(  \Omega\right)  $, we have%
\[
\int_{\Omega}A\left\vert \nabla u\right\vert ^{p}dx\geq\int_{\Omega
}A\left\vert \frac{\overrightarrow{X}}{\left\vert \overrightarrow
{X}\right\vert }\cdot\nabla u\right\vert ^{p}dx\geq\int_{\Omega}\left(
-\operatorname{div}\left(  A\left\vert \overrightarrow{X}\right\vert
^{p-2}\overrightarrow{X}\right)  -\left(  p-1\right)  A\left\vert
\overrightarrow{X}\right\vert ^{p}\right)  \left\vert u\right\vert ^{p}dx
\]
and%
\begin{align*}
&  \left(  \int_{\Omega}A\left\vert \nabla u\right\vert ^{p}dx\right)
^{\frac{1}{p}}\left(  \int_{\Omega}A\left\vert \overrightarrow{X}\right\vert
^{p}\left\vert u\right\vert ^{p}dx\right)  ^{\frac{p-1}{p}}\\
&  \geq\left(  \int_{\Omega}A\left\vert \frac{\overrightarrow{X}}{\left\vert
\overrightarrow{X}\right\vert }\cdot\nabla u\right\vert ^{p}dx\right)
^{\frac{1}{p}}\left(  \int_{\Omega}A\left\vert \overrightarrow{X}\right\vert
^{p}\left\vert u\right\vert ^{p}dx\right)  ^{\frac{p-1}{p}}\\
&  \geq-\frac{1}{p}\int_{\Omega}\operatorname{div}\left(  A\left\vert
\overrightarrow{X}\right\vert ^{p-2}\overrightarrow{X}\right)  \left\vert
u\right\vert ^{p}dx.
\end{align*}

\end{theorem}

It is worth noting that as simple applications of our main results, we obtain
the following $L^{p}$-Hardy inequalities and $L^{p}$-CKN inequalities with
$p$-Bessel pairs:

\begin{theorem}
\label{T5}Let $N\geq1$, $p>1$, $0<R\leq\infty$, $V\geq0$ and $W$ be smooth
functions on $\left(  0,R\right)  $. If $\left(  r^{N-1}V,r^{N-1}W\right)  $
is a $p$-Bessel pair on $\left(  0,R\right)  $, that is, the ODE $\left(
r^{N-1}V\left(  r\right)  \left\vert y^{\prime}\right\vert ^{p-2}y^{\prime
}\right)  ^{\prime}+r^{N-1}W\left(  r\right)  \left\vert y\right\vert
^{p-2}y=0$ has a positive solution $\varphi$ on $\left(  0,R\right)  $, then
for all $u\in C_{0}^{\infty}(B_{R}\setminus\{0\})$:%
\[
{\int\limits_{B_{R}}}V\left(  \left\vert x\right\vert \right)  \left\vert
\nabla u\right\vert ^{p}\mathrm{dx}\geq{\int\limits_{B_{R}}}V\left(
\left\vert x\right\vert \right)  \left\vert \frac{x}{\left\vert x\right\vert
}\cdot\nabla u\right\vert ^{p}\mathrm{dx}\geq{\int\limits_{B_{R}}}W\left(
\left\vert x\right\vert \right)  \left\vert u\right\vert ^{p}\mathrm{dx}%
\]
and%
\begin{align*}
&  \left(  {\int\limits_{B_{R}}}V\left(  \left\vert x\right\vert \right)
\left\vert \nabla u\right\vert ^{p}\mathrm{dx}\right)  ^{\frac{1}{p}}\left(
{\int\limits_{B_{R}}}\left\vert \frac{\varphi^{\prime}}{\varphi}\right\vert
^{p}V\left(  \left\vert x\right\vert \right)  \left\vert u\right\vert
^{p}\mathrm{dx}\right)  ^{\frac{p-1}{p}}\\
&  \geq\left(  {\int\limits_{B_{R}}}V\left(  \left\vert x\right\vert \right)
\left\vert \frac{x}{\left\vert x\right\vert }\cdot\nabla u\right\vert
^{p}\mathrm{dx}\right)  ^{\frac{1}{p}}\left(  {\int\limits_{B_{R}}}\left\vert
\frac{\varphi^{\prime}}{\varphi}\right\vert ^{p}V\left(  \left\vert
x\right\vert \right)  \left\vert u\right\vert ^{p}\mathrm{dx}\right)
^{\frac{p-1}{p}}\\
&  \geq\frac{1}{p}{\int\limits_{B_{R}}}\left[  W\left(  \left\vert
x\right\vert \right)  +\left(  p-1\right)  \left\vert \frac{\varphi^{\prime}%
}{\varphi}\right\vert ^{p}V\left(  \left\vert x\right\vert \right)  \right]
\left\vert u\right\vert ^{p}\mathrm{dx}\text{.}%
\end{align*}

\end{theorem}

We note that the weights in Theorems \ref{T1}, \ref{T2}, \ref{T3} and \ref{T4}
are not necessarily radial. Therefore, our identities and inequalities can be
applied to derive non-radial weights $L^{p}$-Hardy type inequalities and
$L^{p}$-CKN type inequalities. For instance, we can deduce the following
$L^{p}$-Hardy type inequalities and $L^{p}$-CKN type inequalities with
monomial weights:

\begin{theorem}
\label{T5.1}Let $N\geq1$, $p>1$, $0<R\leq\infty$, $V\geq0$ and $W$ be smooth
functions on $\left(  0,R\right)  $. If $\left(  r^{N+\left\vert P\right\vert
-1}V,r^{N+\left\vert P\right\vert -1}W\right)  $ is a $p$-Bessel pair on
$\left(  0,R\right)  $, that is, the ODE $\left(  r^{N+\left\vert P\right\vert
-1}V\left(  r\right)  \left\vert y^{\prime}\right\vert ^{p-2}y^{\prime
}\right)  ^{\prime}+r^{N+\left\vert P\right\vert -1}W\left(  r\right)
\left\vert y\right\vert ^{p-2}y=0$ has a positive solution $\varphi$ on
$\left(  0,R\right)  $, then for all $u\in C_{0}^{\infty}(B_{R}^{\ast
}\setminus\{0\})$:%
\[
{\int\limits_{B_{R}^{\ast}}}V\left(  \left\vert x\right\vert \right)
\left\vert \nabla u\right\vert ^{p}x^{P}\mathrm{dx}\geq{\int\limits_{B_{R}%
^{\ast}}}V\left(  \left\vert x\right\vert \right)  \left\vert \frac
{x}{\left\vert x\right\vert }\cdot\nabla u\right\vert ^{p}x^{P}\mathrm{dx}%
\geq{\int\limits_{B_{R}^{\ast}}}W\left(  \left\vert x\right\vert \right)
\left\vert u\right\vert ^{p}x^{P}\mathrm{dx}%
\]
and%
\begin{align*}
&  \left(  {\int\limits_{B_{R}^{\ast}}}V\left(  \left\vert x\right\vert
\right)  \left\vert \nabla u\right\vert ^{p}x^{P}\mathrm{dx}\right)
^{\frac{1}{p}}\left(  {\int\limits_{B_{R}^{\ast}}}\left\vert \frac
{\varphi^{\prime}}{\varphi}\right\vert ^{p}V\left(  \left\vert x\right\vert
\right)  \left\vert u\right\vert ^{p}x^{P}\mathrm{dx}\right)  ^{\frac{p-1}{p}%
}\\
&  \geq\left(  {\int\limits_{B_{R}^{\ast}}}V\left(  \left\vert x\right\vert
\right)  \left\vert \frac{x}{\left\vert x\right\vert }\cdot\nabla u\right\vert
^{p}x^{P}\mathrm{dx}\right)  ^{\frac{1}{p}}\left(  {\int\limits_{B_{R}^{\ast}%
}}\left\vert \frac{\varphi^{\prime}}{\varphi}\right\vert ^{p}V\left(
\left\vert x\right\vert \right)  \left\vert u\right\vert ^{p}x^{P}%
\mathrm{dx}\right)  ^{\frac{p-1}{p}}\\
&  \geq\frac{1}{p}{\int\limits_{B_{R}^{\ast}}}\left[  W\left(  \left\vert
x\right\vert \right)  +\left(  p-1\right)  \left\vert \frac{\varphi^{\prime}%
}{\varphi}\right\vert ^{p}V\left(  \left\vert x\right\vert \right)  \right]
\left\vert u\right\vert ^{p}x^{P}\mathrm{dx}\text{.}%
\end{align*}
Here $x^{P}=\left\vert x_{1}\right\vert ^{P_{1}}...\left\vert x_{N}\right\vert
^{P_{N}}$, $P_{1}\geq0,...,$ $P_{N}\geq0$, is the monomial weight, $\left\vert
P\right\vert =P_{1}+...+P_{N}$, $\mathbb{R}_{\ast}^{N}=\left\{  \left(
x_{1},...,x_{N}\right)  \in\mathbb{R}^{N}:x_{i}>0\text{ whenever }%
P_{i}>0\right\}  $, and $B_{R}^{\ast}=B_{R}\cap\mathbb{R}_{\ast}^{N}$.
\end{theorem}

As an application of Theorem \ref{T5} and Theorem \ref{T5.1}, we can derive as
many $L^{p}$-Hardy inequalities and $L^{p}$-CKN inequalities as we can form
$p$-Bessel pairs. We also note that $p$-Bessel pair is a $L^{p}$ version of
the Bessel pair \cite{GM1}. It has been used in \cite{DLL22} to set up several
$L^{p}$-Hardy identities and inequalities.

We can also derive the following $L^{2}$-CKN inequalities using our main results:

\begin{corollary}
\label{2CKN} For $u\in C_{0}^{\infty}\left(  \mathbb{R}^{N}\setminus\left\{
0\right\}  \right)  :$

\begin{enumerate}
\item If $b+1-a>0$ and $b\leq\dfrac{N-2}{2}$, then%
\begin{equation}
\left(  \int_{\mathbb{R}^{N}}\dfrac{|\nabla u|^{2}}{|x|^{2b}}dx\right)
^{\frac{1}{2}}\left(  \int_{\mathbb{R}^{N}}\dfrac{|u|^{2}}{|x|^{2a}}dx\right)
^{\frac{1}{2}}\geq\left\vert \dfrac{N-a-b-1}{2}\right\vert \left(
\int_{\mathbb{R}^{N}}\dfrac{|u|^{2}}{|x|^{a+b+1}}dx\right)  . \label{2CKN1}%
\end{equation}
This equality happens iff $u(x)=\alpha\exp\left(  -\dfrac{\beta}%
{b+1-a}|x|^{b+1-a}\right)  $ for some $\alpha\in\mathbb{R},\beta>0$.

\item If $b+1-a<0$ and $b\geq\dfrac{N-2}{2}$, then%
\begin{equation}
\left(  \int_{\mathbb{R}^{N}}\dfrac{|\nabla u|^{2}}{|x|^{2b}}dx\right)
^{\frac{1}{2}}\left(  \int_{\mathbb{R}^{N}}\dfrac{|u|^{2}}{|x|^{2a}}dx\right)
^{\frac{1}{2}}\geq\left\vert \dfrac{a+b+1-N}{2}\right\vert \left(
\int_{\mathbb{R}^{N}}\dfrac{|u|^{2}}{|x|^{a+b+1}}dx\right)  . \label{2CKN2}%
\end{equation}
This equality happens iff $u(x)=\alpha\exp\left(  \dfrac{\beta}{b+1-a}%
|x|^{b+1-a}\right)  $ for some $\alpha\in\mathbb{R},\beta>0$.

\item If $b+1-a<0$ and $b\leq\dfrac{N-2}{2}$, then%
\begin{equation}
\left(  \int_{\mathbb{R}^{N}}\dfrac{|\nabla u|^{2}}{|x|^{2b}}dx\right)
^{\frac{1}{2}}\left(  \int_{\mathbb{R}^{N}}\dfrac{|u|^{2}}{|x|^{2a}}dx\right)
^{\frac{1}{2}}\geq\left\vert \dfrac{N-3b+a-3}{2}\right\vert \left(
\int_{\mathbb{R}^{N}}\dfrac{|u|^{2}}{|x|^{a+b+1}}dx\right)  . \label{2CKN3}%
\end{equation}
This equality happens iff $u(x)=\alpha\left\vert x\right\vert ^{2b+2-N}%
\exp\left(  \dfrac{\beta}{b+1-a}|x|^{b+1-a}\right)  $ for some $\alpha
\in\mathbb{R},\beta>0$.

\item If $b+1-a>0$ and $b\geq\dfrac{N-2}{2}$, then%
\begin{equation}
\left(  \int_{\mathbb{R}^{N}}\dfrac{|\nabla u|^{2}}{|x|^{2b}}dx\right)
^{\frac{1}{2}}\left(  \int_{\mathbb{R}^{N}}\dfrac{|u|^{2}}{|x|^{2a}}dx\right)
^{\frac{1}{2}}\geq\left\vert \dfrac{N-3b+a-3}{2}\right\vert \left(
\int_{\mathbb{R}^{N}}\dfrac{|u|^{2}}{|x|^{a+b+1}}dx\right)  . \label{2CKN4}%
\end{equation}
This equality happens iff $u(x)=\alpha\left\vert x\right\vert ^{2b+2-N}%
\exp\left(  -\dfrac{\beta}{b+1-a}|x|^{b+1-a}\right)  $ for some $\alpha
\in\mathbb{R},\beta>0$.
\end{enumerate}
\end{corollary}

It is worth noting that Corollary \eqref{2CKN} contains some important
inequalities in the literature such as the Heisenberg Uncertainty Principle
($a=-1,~b=0$), the Hydrogen Uncertainty Principle ($a=b=0$), the Hardy
inequalities ($a=1,~b=0$), etc. The sharp constants of the above $L^{2}$-CKN
inequalities have been investigated in \cite{CC09} using some technical tools
such as the Emden-Fowler transformation, the spherical harmonics decomposition
and the Kelvin-type transform. See also \cite{Cos08}. We also refer the
interested reader to \cite{CFL21} for a simple proof of these results. In this
paper, we are able to derive the exact remainders of these results as simple
applications of our main results.

In \cite{CFLL23}, the stability of the $L^{2}$-CKN inequality (\ref{2CKN1})
has also been investigated. In particular, the authors provided in
\cite{CFLL23} a simple approach to establish the sharp stability with explicit
optimal constants of the Heisenberg Uncertainty Principle.

It is also worthy to mention that the stability of functional and geometric
inequalities has been the topic of extensive and intensive studies in the last
few years. It has been motivated by a question raised by Brezis and Lieb in
\cite{BL85} and some results on the stability of the $L^{2}$-Sobolev
inequalities by Bianchi and Egnell in \cite{BE91}. The interested reader is
referred to more extensive development in this direction \cite{BWW03, BDNS20, BDNS20a, Car, CF13, CLT22, CLT23, CFW13, CFMP09,
DeNK23, DEFFL, DN21, Fat21, FIL16, FJ15, FJ17,  FMP13, FN19, Frank, Kon22, LW00, McV21}, to name just a few.

Our next goal of this paper is to use our main results on the remainders to
establish the stability results of certain $L^{p}$-CKN inequalities. More
precisely, we will first study the stability of the $L^{2}$-CKN inequality
(\ref{2CKN4}) and prove that

\begin{theorem}
\label{T6} Let $\frac{N-2}{2}<b\leq N-2$ and $N\left(  b-a+3\right)  =2\left(
3b-a+3\right)  $. There exists a universal constant $C(N,a,b)>0$ such that
\begin{align*}
&  \left(  \int_{\mathbb{R}^{N}}\dfrac{|\nabla u|^{2}}{|x|^{2b}}dx\right)
^{1/2}\left(  \int_{\mathbb{R}^{N}}\dfrac{|u|^{2}}{|x|^{2a}}dx\right)
^{1/2}-\frac{3b-a-N+3}{2}\int_{\mathbb{R}^{N}}\dfrac{|u|^{2}}{|x|^{a+b+1}}dx\\
&  \geq C(N,a,b)\inf_{c\mathbb{\in R},\lambda>0}\int_{\mathbb{R}^{N}}%
\dfrac{\left\vert u-c|x|^{2b+2-N}e^{-\frac{\lambda}{b+1-a}|x|^{b+1-a}%
}\right\vert ^{2}}{|x|^{a+b+1}}dx.
\end{align*}

\end{theorem}

In the same line of thought, we will also establish the following $L^{p}$-CKN
inequalities with exact remainders, as a consequence of our main result:

\begin{corollary}
\label{CKN} Let $N\geq1,~p>1$. Then for any $u\in C_{0}^{\infty}%
(\mathbb{R}^{N}\setminus\{0\})$, there holds
\[
\left(  \int_{\mathbb{R}^{N}}\dfrac{|\nabla u|^{p}}{|x|^{pb}}dx\right)
^{\frac{1}{p}}\left(  \int_{\mathbb{R}^{N}}\dfrac{|u|^{p}}{|x|^{pa}}dx\right)
^{\frac{p-1}{p}}\geq\frac{\left\vert N-1-\left(  p-1\right)  a-b\right\vert
}{p}\int_{\mathbb{R}^{N}}\dfrac{|u|^{p}}{|x|^{(p-1)a+b+1}}dx.
\]
Also,

\begin{enumerate}
\item If $b+1-a>0$ and $b\leq\frac{N-p}{p}$, then the constant $\frac
{N-1-\left(  p-1\right)  a-b}{p}$ is sharp and can be attained only by the
functions of the form $u(x)=D\exp(\frac{t|x|^{b+1-a}}{b+1-a}),$ $t<0$.

\item If $b+1-a<0$ and $b\geq\frac{N-p}{p}$, then the constant $\frac
{1+(p-1)a+b-N}{p}$ is sharp and can be attained only by the functions of the
form $u(x)=D\exp(\frac{t|x|^{b+1-a}}{b+1-a}),$ $t>0$.
\end{enumerate}
\end{corollary}

Using the explicit form on the remainder and Lemma \ref{L1}, we then
investigate their stability and prove the following result

\begin{theorem}
\label{T8} Let $p\geq2$, $0\leq b<\frac{N-p}{p}$, $a\leq\frac{Nb}{N-p}$ and
$(p-1)a+b+1=\frac{pbN}{N-p}$. There exists a universal constant $C(N,p,a,b)>0$
such that for all $u\in C_{0}^{\infty}(\mathbb{R}^{N}\setminus\{0\}):$
\begin{align*}
&  \left(  \int_{\mathbb{R}^{N}}\dfrac{|\nabla u|^{p}}{|x|^{pb}}dx\right)
^{\frac{1}{p}}\left(  \int_{\mathbb{R}^{N}}\dfrac{|u|^{p}}{|x|^{pa}}dx\right)
^{\frac{p-1}{p}}-\dfrac{N-1-\left(  p-1\right)  a-b}{p}\int_{\mathbb{R}^{N}%
}\dfrac{|u|^{p}}{|x|^{(p-1)a+b+1}}dx\\
&  \geq C(N,p,a,b)\inf_{c\mathbb{\in R},\lambda>0}\int_{\mathbb{R}^{N}}%
\dfrac{\left\vert u-ce^{-\frac{\lambda}{b+1-a}|x|^{b+1-a}}\right\vert ^{p}%
}{|x|^{(p-1)a+b+1}}dx.
\end{align*}

\end{theorem}

The paper is organized as follows: In section 2, we will give a proof of our
main result (Theorem \ref{T1}). In section 3, we will use our main results to
derive several $L^{2}$-Hardy identities and inequalities and $L^{2}%
$-Caffarelli-Kohn-Nirenberg identities and inequalities, as well as their
stabilities. In section 4, we present a proof of Theorem \ref{T5}, as well as
many other $L^{p}$-Caffarelli-Kohn-Nirenberg inequalities and their stabilities.

\section{Proofs of Theorem \ref{T1}}

The main purpose of this section is to give a proof of our first general
$L^{p}$ identity.

\begin{proof}
[Proof of Theorem \ref{T1}]Using the Divergence Theorem and the definition of
$\mathcal{R}_{p}$, we get%
\begin{align*}
&  -\int_{\Omega}\operatorname{div}\left(  A\left\vert \overrightarrow
{X}\right\vert ^{p-2}\overrightarrow{X}\right)  \left\vert u\right\vert
^{p}dx\\
&  =\int_{\Omega}A\left\vert \overrightarrow{X}\right\vert ^{p-2}%
\overrightarrow{X}\cdot\nabla\left\vert u\right\vert ^{p}dx\\
&  =p\int_{\Omega}A\frac{1}{\alpha}\left\vert u\overrightarrow{X}\right\vert
^{p-2}u\overrightarrow{X}\cdot\alpha\nabla udx\\
&  =\left\vert \alpha\right\vert ^{p}\int_{\Omega}A\left\vert \nabla
u\right\vert ^{p}dx+\frac{\left(  p-1\right)  }{\left\vert \alpha\right\vert
^{\frac{p}{p-1}}}\int_{\Omega}A\left\vert \overrightarrow{X}\right\vert
^{p}\left\vert u\right\vert ^{p}dx-\int_{\Omega}A\mathcal{R}_{p}\left(
\frac{1}{\alpha^{\frac{1}{p-1}}}u\overrightarrow{X},\alpha\nabla u\right)  dx.
\end{align*}
Similarly%
\begin{align*}
&  -\int_{\Omega}\operatorname{div}\left(  A\left\vert \overrightarrow
{X}\right\vert ^{p-2}\overrightarrow{X}\right)  \left\vert u\right\vert
^{p}dx\\
&  =\int_{\Omega}A\left\vert \overrightarrow{X}\right\vert ^{p-2}%
\overrightarrow{X}\cdot\nabla\left\vert u\right\vert ^{p}dx\\
&  =p\int_{\Omega}A\frac{1}{\alpha}\left\vert u\overrightarrow{X}\right\vert
^{p-2}u\left\vert \overrightarrow{X}\right\vert \left(  \alpha\frac
{\overrightarrow{X}}{\left\vert \overrightarrow{X}\right\vert }\cdot\nabla
u\right)  dx\\
&  =\left\vert \alpha\right\vert ^{p}\int_{\Omega}A\left\vert \frac
{\overrightarrow{X}}{\left\vert \overrightarrow{X}\right\vert }\cdot\nabla
u\right\vert ^{p}dx+\frac{\left(  p-1\right)  }{\left\vert \alpha\right\vert
^{\frac{p}{p-1}}}\int_{\Omega}A\left\vert \overrightarrow{X}\right\vert
^{p}\left\vert u\right\vert ^{p}dx\\
&  \hspace{1in}-\int_{\Omega}A\mathcal{R}_{p}\left(  \frac{1}{\alpha^{\frac
{1}{p-1}}}u\left\vert \overrightarrow{X}\right\vert ,\alpha\frac
{\overrightarrow{X}}{\left\vert \overrightarrow{X}\right\vert }\cdot\nabla
u\right)  dx.
\end{align*}

\end{proof}

\section{$L^{2}$-Hardy identities and inequalities and $L^{2}$%
-Caffarelli-Kohn-Nirenberg identities and inequalities}

When $p=2$, noting that $R_{2}\left(  \overrightarrow{a},\overrightarrow
{b}\right)  =\left\Vert \overrightarrow{a}-\overrightarrow{b}\right\Vert ^{2}%
$, we obtain the following identities and inequalities from our main results:

\begin{theorem}
\label{T2.1}Let $\Omega$ be an open set in $\mathbb{R}^{N}$, $N\geq1$,
$\alpha>0$, $A\in C^{1}\left(  \Omega\right)  $ and $\overrightarrow{X}\in
C^{1}\left(  \Omega,\mathbb{R}^{N}\right)  $. Then for any $u\in C_{0}%
^{1}\left(  \Omega\right)  $, we have%
\begin{align*}
&  \left\vert \alpha\right\vert ^{2}\int_{\Omega}A\left\vert \nabla
u\right\vert ^{2}dx+\frac{1}{\left\vert \alpha\right\vert ^{2}}\int_{\Omega
}A\left\vert \overrightarrow{X}\right\vert ^{2}\left\vert u\right\vert
^{2}dx\\
&  =-\int_{\Omega}\operatorname{div}\left(  A\overrightarrow{X}\right)
\left\vert u\right\vert ^{2}dx+\int_{\Omega}A\left\vert \alpha\nabla
u-\frac{1}{\alpha}u\overrightarrow{X}\right\vert ^{2}dx,
\end{align*}%
\begin{align*}
&  \left\vert \alpha\right\vert ^{2}\int_{\Omega}A\left\vert \frac
{\overrightarrow{X}}{\left\vert \overrightarrow{X}\right\vert }\cdot\nabla
u\right\vert ^{2}dx+\frac{1}{\left\vert \alpha\right\vert ^{2}}\int_{\Omega
}A\left\vert \overrightarrow{X}\right\vert ^{2}\left\vert u\right\vert
^{2}dx\\
&  =-\int_{\Omega}\operatorname{div}\left(  A\overrightarrow{X}\right)
\left\vert u\right\vert ^{2}dx+\int_{\Omega}A\left\vert \alpha\frac
{\overrightarrow{X}}{\left\vert \overrightarrow{X}\right\vert }\cdot\nabla
u-\frac{1}{\alpha}u\left\vert \overrightarrow{X}\right\vert \right\vert
^{2}dx.
\end{align*}
As consequences,%
\begin{align*}
\int_{\Omega}A\left\vert \nabla u\right\vert ^{2}dx  &  =\int_{\Omega}\left(
-\operatorname{div}\left(  A\overrightarrow{X}\right)  -A\left\vert
\overrightarrow{X}\right\vert ^{2}\right)  \left\vert u\right\vert ^{2}%
dx+\int_{\Omega}A\left\vert \nabla u-u\overrightarrow{X}\right\vert ^{2}dx\\
\int_{\Omega}A\left\vert \frac{\overrightarrow{X}}{\left\vert \overrightarrow
{X}\right\vert }\cdot\nabla u\right\vert ^{2}dx  &  =\int_{\Omega}\left(
-\operatorname{div}\left(  A\overrightarrow{X}\right)  -A\left\vert
\overrightarrow{X}\right\vert ^{2}\right)  \left\vert u\right\vert ^{2}%
dx+\int_{\Omega}A\left\vert \frac{\overrightarrow{X}}{\left\vert
\overrightarrow{X}\right\vert }\cdot\nabla u-u\left\vert \overrightarrow
{X}\right\vert \right\vert ^{2}dx.
\end{align*}
Also, if $A\geq0$, then%
\begin{align*}
&  \left(  \int_{\Omega}A\left\vert \nabla u\right\vert ^{2}dx\right)
^{\frac{1}{2}}\left(  \int_{\Omega}A\left\vert \overrightarrow{X}\right\vert
^{2}\left\vert u\right\vert ^{2}dx\right)  ^{\frac{1}{2}}+\frac{1}{2}%
\int_{\Omega}\operatorname{div}\left(  A\overrightarrow{X}\right)  \left\vert
u\right\vert ^{2}dx\\
&  =\frac{1}{2}\int_{\Omega}A\left\vert \left(  \frac{\int_{\Omega}A\left\vert
\overrightarrow{X}\right\vert ^{2}\left\vert u\right\vert ^{2}dx}{\int
_{\Omega}A\left\vert \nabla u\right\vert ^{2}dx}\right)  ^{\frac{1}{4}}\nabla
u-\left(  \frac{\int_{\Omega}A\left\vert \nabla u\right\vert ^{2}dx}%
{\int_{\Omega}A\left\vert \overrightarrow{X}\right\vert ^{2}\left\vert
u\right\vert ^{2}dx}\right)  ^{\frac{1}{4}}u\overrightarrow{X}\right\vert
^{2}dx
\end{align*}
and%
\begin{align*}
&  \left(  \int_{\Omega}A\left\vert \frac{\overrightarrow{X}}{\left\vert
\overrightarrow{X}\right\vert }\cdot\nabla u\right\vert ^{2}dx\right)
^{\frac{1}{2}}\left(  \int_{\Omega}A\left\vert \overrightarrow{X}\right\vert
^{2}\left\vert u\right\vert ^{2}dx\right)  ^{\frac{1}{2}}+\frac{1}{2}%
\int_{\Omega}\operatorname{div}\left(  A\overrightarrow{X}\right)  \left\vert
u\right\vert ^{2}dx\\
&  =\frac{1}{2}\int_{\Omega}A\left\vert \left(  \frac{\int_{\Omega}A\left\vert
\overrightarrow{X}\right\vert ^{2}\left\vert u\right\vert ^{2}dx}{\int
_{\Omega}A\left\vert \frac{\overrightarrow{X}}{\left\vert \overrightarrow
{X}\right\vert }\cdot\nabla u\right\vert ^{2}dx}\right)  ^{\frac{1}{4}}%
\frac{\overrightarrow{X}}{\left\vert \overrightarrow{X}\right\vert }%
\cdot\nabla u-\left(  \frac{\int_{\Omega}A\left\vert \frac{\overrightarrow{X}%
}{\left\vert \overrightarrow{X}\right\vert }\cdot\nabla u\right\vert ^{2}%
dx}{\int_{\Omega}A\left\vert \overrightarrow{X}\right\vert ^{2}\left\vert
u\right\vert ^{2}dx}\right)  ^{\frac{1}{4}}u\left\vert \overrightarrow
{X}\right\vert \right\vert ^{2}dx.
\end{align*}

\end{theorem}

We will next provide here some consequences.

\begin{corollary}
[Hardy inequalities and CKN inequalities with Bessel pairs]\label{c5}%
\textit{Let }$0<R\leq\infty$\textit{, }$V\geq0$\textit{ and }$W$\textit{ be
}$C^{1}$-\textit{functions on }$\left(  0,R\right)  $. Assume that $\left(
r^{N-1}V,r^{N-1}W\right)  $ is a Bessel pair on $\left(  0,R\right)  $, that
is, there exists a positive function $\varphi$ such that
\[
\left(  r^{N-1}V\varphi^{\prime}\right)  ^{\prime}+r^{N-1}W\varphi=0\text{ on
}\left(  0,R\right)  \text{.}%
\]
Then we have for all $u\in C_{0}^{\infty}\left(  B_{R}\setminus\left\{
0\right\}  \right)  \setminus\left\{  0\right\}  $ that
\begin{align*}
{\int\limits_{B_{R}}}V\left(  \left\vert x\right\vert \right)  \left\vert
\nabla u\right\vert ^{2}\mathrm{dx}  &  ={\int\limits_{B_{R}}}W\left(
\left\vert x\right\vert \right)  \left\vert u\right\vert ^{2}\mathrm{dx}%
+{\int\limits_{B_{R}}}V\left(  \left\vert x\right\vert \right)  \left\vert
\nabla u-\frac{\varphi^{\prime}\left(  \left\vert x\right\vert \right)
}{\varphi\left(  \left\vert x\right\vert \right)  }u\frac{x}{\left\vert
x\right\vert }\right\vert ^{2}\mathrm{dx}\\
&  ={\int\limits_{B_{R}}}W\left(  \left\vert x\right\vert \right)  \left\vert
u\right\vert ^{2}\mathrm{dx}+{\int\limits_{B_{R}}}V\left(  \left\vert
x\right\vert \right)  \varphi^{2}\left(  \left\vert x\right\vert \right)
\left\vert \nabla\left(  \frac{u\left(  x\right)  }{\varphi\left(  \left\vert
x\right\vert \right)  }\right)  \right\vert ^{2}\mathrm{dx},
\end{align*}%
\begin{align*}
{\int\limits_{B_{R}}}V\left(  \left\vert x\right\vert \right)  \left\vert
\frac{x}{\left\vert x\right\vert }\cdot\nabla u\right\vert ^{2}\mathrm{dx}  &
={\int\limits_{B_{R}}}W\left(  \left\vert x\right\vert \right)  \left\vert
u\right\vert ^{2}\mathrm{dx}+{\int\limits_{B_{R}}}V\left(  \left\vert
x\right\vert \right)  \left\vert \frac{x}{\left\vert x\right\vert }\cdot\nabla
u-\frac{\varphi^{\prime}\left(  \left\vert x\right\vert \right)  }%
{\varphi\left(  \left\vert x\right\vert \right)  }u\right\vert ^{2}%
\mathrm{dx}\\
&  ={\int\limits_{B_{R}}}W\left(  \left\vert x\right\vert \right)  \left\vert
u\right\vert ^{2}\mathrm{dx}+{\int\limits_{B_{R}}}V\left(  \left\vert
x\right\vert \right)  \varphi^{2}\left(  \left\vert x\right\vert \right)
\left\vert \frac{x}{\left\vert x\right\vert }\cdot\nabla\left(  \frac{u\left(
x\right)  }{\varphi\left(  \left\vert x\right\vert \right)  }\right)
\right\vert ^{2}\mathrm{dx},
\end{align*}%
\begin{align*}
&  \left(  {\int\limits_{B_{R}}}V\left(  \left\vert x\right\vert \right)
\left\vert \nabla u\right\vert ^{2}\mathrm{dx}\right)  ^{\frac{1}{2}}\left(
{\int\limits_{B_{R}}}\left(  \frac{\varphi^{\prime}\left(  \left\vert
x\right\vert \right)  }{\varphi\left(  \left\vert x\right\vert \right)
}\right)  ^{2}V\left(  \left\vert x\right\vert \right)  \left\vert
u\right\vert ^{2}\mathrm{dx}\right)  ^{\frac{1}{2}}\\
&  =\frac{1}{2}{\int\limits_{B_{R}}}\left[  W\left(  \left\vert x\right\vert
\right)  +\left(  \frac{\varphi^{\prime}\left(  \left\vert x\right\vert
\right)  }{\varphi\left(  \left\vert x\right\vert \right)  }\right)
^{2}V\left(  \left\vert x\right\vert \right)  \right]  \left\vert u\right\vert
^{2}\mathrm{dx}\\
&  +\frac{1}{2}{\int\limits_{B_{R}}}V\left(  \left\vert x\right\vert \right)
\left\vert \frac{\left\Vert \frac{\varphi^{\prime}}{\varphi}\sqrt
{V}u\right\Vert _{2}^{\frac{1}{2}}}{\left\Vert \sqrt{V}\left\vert \nabla
u\right\vert \right\Vert _{2}^{\frac{1}{2}}}\nabla u-\frac{\left\Vert \sqrt
{V}\left\vert \nabla u\right\vert \right\Vert _{2}^{\frac{1}{2}}}{\left\Vert
\frac{\varphi^{\prime}}{\varphi}\sqrt{V}u\right\Vert _{2}^{\frac{1}{2}}}%
\frac{\varphi^{\prime}\left(  \left\vert x\right\vert \right)  }%
{\varphi\left(  \left\vert x\right\vert \right)  }u\frac{x}{\left\vert
x\right\vert }\right\vert ^{2}\mathrm{dx},
\end{align*}
and%
\begin{align*}
&  \left(  {\int\limits_{B_{R}}}V\left(  \left\vert x\right\vert \right)
\left\vert \frac{x}{\left\vert x\right\vert }\cdot\nabla u\right\vert
^{2}\mathrm{dx}\right)  ^{\frac{1}{2}}\left(  {\int\limits_{B_{R}}}\left(
\frac{\varphi^{\prime}\left(  \left\vert x\right\vert \right)  }%
{\varphi\left(  \left\vert x\right\vert \right)  }\right)  ^{2}V\left(
\left\vert x\right\vert \right)  \left\vert u\right\vert ^{2}\mathrm{dx}%
\right)  ^{\frac{1}{2}}\\
&  =\frac{1}{2}{\int\limits_{B_{R}}}\left[  W\left(  \left\vert x\right\vert
\right)  +\left(  \frac{\varphi^{\prime}\left(  \left\vert x\right\vert
\right)  }{\varphi\left(  \left\vert x\right\vert \right)  }\right)
^{2}V\left(  \left\vert x\right\vert \right)  \right]  \left\vert u\right\vert
^{2}\mathrm{dx}\\
&  +\frac{1}{2}{\int\limits_{B_{R}}}V\left(  \left\vert x\right\vert \right)
\left\vert \frac{\left\Vert \frac{\varphi^{\prime}}{\varphi}\sqrt
{V}u\right\Vert _{2}^{\frac{1}{2}}}{\left\Vert \sqrt{V}\left\vert \frac
{x}{\left\vert x\right\vert }\cdot\nabla u\right\vert \right\Vert _{2}%
^{\frac{1}{2}}}\frac{x}{\left\vert x\right\vert }\cdot\nabla u-\frac
{\left\Vert \sqrt{V}\left\vert \frac{x}{\left\vert x\right\vert }\cdot\nabla
u\right\vert \right\Vert _{2}^{\frac{1}{2}}}{\left\Vert \frac{\varphi^{\prime
}}{\varphi}\sqrt{V}u\right\Vert _{2}^{\frac{1}{2}}}\frac{\varphi^{\prime
}\left(  \left\vert x\right\vert \right)  }{\varphi\left(  \left\vert
x\right\vert \right)  }u\right\vert ^{2}\mathrm{dx}.
\end{align*}
Therefore%
\[
{\int\limits_{B_{R}}}V\left(  \left\vert x\right\vert \right)  \left\vert
\nabla u\right\vert ^{2}\mathrm{dx}\geq{\int\limits_{B_{R}}}V\left(
\left\vert x\right\vert \right)  \left\vert \frac{x}{\left\vert x\right\vert
}\cdot\nabla u\right\vert ^{2}\mathrm{dx}\geq{\int\limits_{B_{R}}}W\left(
\left\vert x\right\vert \right)  \left\vert u\right\vert ^{2}\mathrm{dx}%
\]
and%
\begin{align*}
&  \left(  {\int\limits_{B_{R}}}V\left(  \left\vert x\right\vert \right)
\left\vert \nabla u\right\vert ^{2}\mathrm{dx}\right)  ^{\frac{1}{2}}\left(
{\int\limits_{B_{R}}}\left(  \frac{\varphi^{\prime}\left(  \left\vert
x\right\vert \right)  }{\varphi\left(  \left\vert x\right\vert \right)
}\right)  ^{2}V\left(  \left\vert x\right\vert \right)  \left\vert
u\right\vert ^{2}\mathrm{dx}\right)  ^{\frac{1}{2}}\\
&  \geq\left(  {\int\limits_{B_{R}}}V\left(  \left\vert x\right\vert \right)
\left\vert \frac{x}{\left\vert x\right\vert }\cdot\nabla u\right\vert
^{2}\mathrm{dx}\right)  ^{\frac{1}{2}}\left(  {\int\limits_{B_{R}}}\left(
\frac{\varphi^{\prime}\left(  \left\vert x\right\vert \right)  }%
{\varphi\left(  \left\vert x\right\vert \right)  }\right)  ^{2}V\left(
\left\vert x\right\vert \right)  \left\vert u\right\vert ^{2}\mathrm{dx}%
\right)  ^{\frac{1}{2}}\\
&  \geq\frac{1}{2}{\int\limits_{B_{R}}}\left[  W\left(  \left\vert
x\right\vert \right)  +\left(  \frac{\varphi^{\prime}\left(  \left\vert
x\right\vert \right)  }{\varphi\left(  \left\vert x\right\vert \right)
}\right)  ^{2}V\left(  \left\vert x\right\vert \right)  \right]  \left\vert
u\right\vert ^{2}\mathrm{dx.}%
\end{align*}

\end{corollary}

\begin{proof}
Choose $A=V$ and $\overrightarrow{X}=\frac{\varphi^{\prime}\left(  \left\vert
x\right\vert \right)  }{\varphi\left(  \left\vert x\right\vert \right)  }%
\frac{x}{\left\vert x\right\vert }$. Then%
\begin{align*}
\operatorname{div}\left(  A\overrightarrow{X}\right)   &  =\operatorname{div}%
\left(  V\frac{\varphi^{\prime}\left(  \left\vert x\right\vert \right)
}{\varphi\left(  \left\vert x\right\vert \right)  }\frac{x}{\left\vert
x\right\vert }\right) \\
&  =\nabla\left(  V\frac{\varphi^{\prime}\left(  \left\vert x\right\vert
\right)  }{\varphi\left(  \left\vert x\right\vert \right)  }\right)
\cdot\frac{x}{\left\vert x\right\vert }+V\frac{\varphi^{\prime}\left(
\left\vert x\right\vert \right)  }{\varphi\left(  \left\vert x\right\vert
\right)  }\operatorname{div}\left(  \frac{x}{\left\vert x\right\vert }\right)
\\
&  =V^{\prime}\frac{\varphi^{\prime}\left(  \left\vert x\right\vert \right)
}{\varphi\left(  \left\vert x\right\vert \right)  }+V\frac{\varphi
^{\prime\prime}\left(  \left\vert x\right\vert \right)  }{\varphi\left(
\left\vert x\right\vert \right)  }-V\frac{\left(  \varphi^{\prime}\left(
\left\vert x\right\vert \right)  \right)  ^{2}}{\varphi^{2}\left(  \left\vert
x\right\vert \right)  }+\frac{\left(  N-1\right)  }{\left\vert x\right\vert
}V\frac{\varphi^{\prime}\left(  \left\vert x\right\vert \right)  }%
{\varphi\left(  \left\vert x\right\vert \right)  }\\
&  =-W-V\frac{\left(  \varphi^{\prime}\left(  \left\vert x\right\vert \right)
\right)  ^{2}}{\varphi^{2}\left(  \left\vert x\right\vert \right)  }.
\end{align*}
Therefore, we now can apply Theorem \ref{T2.1} to get the desired results.
\end{proof}

From Corollary \ref{c5}, we have that we can establish as many Hardy
inequalities and CKN inequalities as we can form Bessel pairs. We note that
Bessel pairs have been introduced in \cite{GM1} to study Hardy inequality with
radial weights. Also, many examples and properties of Bessel pairs have been
provided in \cite{GM1}.

We can also derive the following CKN inequalities with exact remainder terms:

\begin{corollary}
Let $b+1-a>0$ and $b\leq\dfrac{N-2}{2}$. Then for $u\in C_{0}^{\infty}\left(
\mathbb{R}^{N}\setminus\left\{  0\right\}  \right)  :$
\begin{align*}
&  \int_{\mathbb{R}^{N}}\dfrac{|\nabla u|^{2}}{|x|^{2b}}dx+\int_{\mathbb{R}%
^{N}}\dfrac{|u|^{2}}{|x|^{2a}}dx-\left(  N-1-a-b\right)  \int_{\mathbb{R}^{N}%
}\dfrac{|u|^{2}}{|x|^{a+b+1}}dx\\
&  =\int_{\mathbb{R}^{N}}\dfrac{1}{|x|^{2b}}\left\vert \nabla\left(
u.e^{\frac{|x|^{b+1-a}}{b+1-a}}\right)  \right\vert ^{2}e^{-\frac
{2|x|^{b+1-a}}{b+1-a}}dx.
\end{align*}

Also, for $u\in C_{0}^{\infty}\left(  \mathbb{R}^{N}\setminus\left\{
0\right\}  \right)  \setminus\left\{  0\right\}  $ and $\lambda=\left(
\frac{\int_{\mathbb{R}^{N}}\dfrac{|u|^{2}}{|x|^{2a}}dx}{\int_{\mathbb{R}^{N}%
}\dfrac{|\nabla u|^{2}}{|x|^{2b}}dx}\right)  ^{\frac{1}{2\left(  b+1-a\right)
}}:$%
\begin{align}
&  \left(  \int_{\mathbb{R}^{N}}\dfrac{|\nabla u|^{2}}{|x|^{2b}}dx\right)
^{\frac{1}{2}}\left(  \int_{\mathbb{R}^{N}}\dfrac{|u|^{2}}{|x|^{2a}}dx\right)
^{\frac{1}{2}}-\left\vert \dfrac{N-a-b-1}{2}\right\vert \left(  \int
_{\mathbb{R}^{N}}\dfrac{|u|^{2}}{|x|^{a+b+1}}dx\right) \nonumber\\
&  =\frac{1}{2}\lambda^{b-a+1}\int_{\mathbb{R}^{N}}\frac{1}{|x|^{2b}%
}\left\vert \nabla\left(  ue^{\frac{|x|^{b+1-a}}{\left(  b+1-a\right)
\lambda^{b-a+1}}}\right)  \right\vert ^{2}e^{-\frac{2|x|^{b+1-a}}{\left(
b+1-a\right)  \lambda^{b-a+1}}}dx \label{iCKN1}%
\end{align}
and%
\[
\left(  \int_{\mathbb{R}^{N}}\dfrac{|\nabla u|^{2}}{|x|^{2b}}dx\right)
^{\frac{1}{2}}\left(  \int_{\mathbb{R}^{N}}\dfrac{|u|^{2}}{|x|^{2a}}dx\right)
^{\frac{1}{2}}\geq\left\vert \dfrac{N-a-b-1}{2}\right\vert \left(
\int_{\mathbb{R}^{N}}\dfrac{|u|^{2}}{|x|^{a+b+1}}dx\right)  .
\]
This equality happens iff $u(x)=\alpha\exp\left(  -\dfrac{\beta}%
{b+1-a}|x|^{b+1-a}\right)  $ for some $\alpha\in\mathbb{R},\beta>0$.
\end{corollary}

\begin{proof}
Let $A=\frac{1}{|x|^{2b}}$ and $\overrightarrow{X}=-|x|^{b-a}\frac{x}{|x|}$ in
Theorem \ref{T2.1}. Note that
\[
-\operatorname{div}\left(  A\overrightarrow{X}\right)  =\operatorname{div}%
\left(  |x|^{-a-b-1}x\right)  =\frac{N-1-a-b}{|x|^{a+b+1}}.
\]

\end{proof}

\begin{corollary}
Let $b+1-a<0$ and $b\geq\dfrac{N-2}{2}$. Then for $u\in C_{0}^{\infty}\left(
\mathbb{R}^{N}\setminus\left\{  0\right\}  \right)  :$
\begin{align*}
&  \int_{\mathbb{R}^{N}}\dfrac{|\nabla u|^{2}}{|x|^{2b}}dx+\int_{\mathbb{R}%
^{N}}\dfrac{|u|^{2}}{|x|^{2a}}dx-\left(  a+b+1-N\right)  \int_{\mathbb{R}^{N}%
}\dfrac{|u|^{2}}{|x|^{a+b+1}}dx\\
&  =\int_{\mathbb{R}^{N}}\dfrac{1}{|x|^{2b}}\left\vert \nabla\left(
u.e^{-\frac{|x|^{b+1-a}}{b+1-a}}\right)  \right\vert ^{2}e^{\frac
{2|x|^{b+1-a}}{b+1-a}}dx.
\end{align*}
Also, for $u\in C_{0}^{\infty}\left(  \mathbb{R}^{N}\setminus\left\{
0\right\}  \right)  \setminus\left\{  0\right\}  $ and $\lambda=\left(
\frac{\int_{\mathbb{R}^{N}}\dfrac{|u|^{2}}{|x|^{2a}}dx}{\int_{\mathbb{R}^{N}%
}\dfrac{|\nabla u|^{2}}{|x|^{2b}}dx}\right)  ^{\frac{1}{2\left(  b+1-a\right)
}}:$%
\begin{align*}
&  \left(  \int_{\mathbb{R}^{N}}\dfrac{|\nabla u|^{2}}{|x|^{2b}}dx\right)
^{\frac{1}{2}}\left(  \int_{\mathbb{R}^{N}}\dfrac{|u|^{2}}{|x|^{2a}}dx\right)
^{\frac{1}{2}}-\left\vert \dfrac{a+b+1-N}{2}\right\vert \left(  \int
_{\mathbb{R}^{N}}\dfrac{|u|^{2}}{|x|^{a+b+1}}dx\right) \\
&  =\frac{1}{2}\lambda^{b-a+1}\int_{\mathbb{R}^{N}}\frac{1}{|x|^{2b}%
}\left\vert \nabla\left(  ue^{-\frac{|x|^{b+1-a}}{\left(  b+1-a\right)
\lambda^{b-a+1}}}\right)  \right\vert ^{2}e^{\frac{2|x|^{b+1-a}}{\left(
b+1-a\right)  \lambda^{b-a+1}}}dx
\end{align*}
and%
\[
\left(  \int_{\mathbb{R}^{N}}\dfrac{|\nabla u|^{2}}{|x|^{2b}}dx\right)
^{\frac{1}{2}}\left(  \int_{\mathbb{R}^{N}}\dfrac{|u|^{2}}{|x|^{2a}}dx\right)
^{\frac{1}{2}}\geq\left\vert \dfrac{a+b+1-N}{2}\right\vert \left(
\int_{\mathbb{R}^{N}}\dfrac{|u|^{2}}{|x|^{a+b+1}}dx\right)  .
\]
This equality happens iff $u(x)=\alpha\exp\left(  \dfrac{\beta}{b+1-a}%
|x|^{b+1-a}\right)  $ for some $\alpha\in\mathbb{R},\beta>0$.
\end{corollary}

\begin{proof}
Let $A=\frac{1}{|x|^{2b}}$ and $\overrightarrow{X}=|x|^{b-a}\frac{x}{|x|}$ in
Theorem \ref{T2.1}. Note that
\[
-\operatorname{div}\left(  A\overrightarrow{X}\right)  =-\operatorname{div}%
\left(  |x|^{-a-b-1}x\right)  =\frac{a+b+1-N}{|x|^{a+b+1}}.
\]

\end{proof}

\begin{corollary}
Let $b+1-a<0$ and $b\leq\dfrac{N-2}{2}$. Then for $u\in C_{0}^{\infty}\left(
\mathbb{R}^{N}\setminus\left\{  0\right\}  \right)  :$
\begin{align*}
&  \int_{\mathbb{R}^{N}}\dfrac{|\nabla u|^{2}}{|x|^{2b}}dx+\int_{\mathbb{R}%
^{N}}\dfrac{|u|^{2}}{|x|^{2a}}dx-\left(  N-3b+a-3\right)  \int_{\mathbb{R}%
^{N}}\dfrac{|u|^{2}}{|x|^{a+b+1}}dx\\
&  =\int_{\mathbb{R}^{N}}\dfrac{1}{|x|^{2N-2b-4}}\left\vert \nabla\left(
u|x|^{N-2b-2}e^{-\frac{|x|^{b+1-a}}{b+1-a}}\right)  \right\vert ^{2}%
e^{\frac{2|x|^{b+1-a}}{b+1-a}}dx.
\end{align*}
Also, for $u\in C_{0}^{\infty}\left(  \mathbb{R}^{N}\setminus\left\{
0\right\}  \right)  \setminus\left\{  0\right\}  $ and $\lambda=\left(
\frac{\int_{\mathbb{R}^{N}}\dfrac{|u|^{2}}{|x|^{2a}}dx}{\int_{\mathbb{R}^{N}%
}\dfrac{|\nabla u|^{2}}{|x|^{2b}}dx}\right)  ^{\frac{1}{2\left(  b+1-a\right)
}}:$%
\begin{align*}
&  \left(  \int_{\mathbb{R}^{N}}\dfrac{|\nabla u|^{2}}{|x|^{2b}}dx\right)
^{\frac{1}{2}}\left(  \int_{\mathbb{R}^{N}}\dfrac{|u|^{2}}{|x|^{2a}}dx\right)
^{\frac{1}{2}}-\left\vert \dfrac{N-3b+a-3}{2}\right\vert \left(
\int_{\mathbb{R}^{N}}\dfrac{|u|^{2}}{|x|^{a+b+1}}dx\right) \\
&  =\frac{1}{2}\lambda^{b-a+1}\int_{\mathbb{R}^{N}}\dfrac{1}{|x|^{2N-2b-4}%
}\left\vert \nabla\left(  u\left\vert x\right\vert ^{N-2b-2}e^{-\frac
{|x|^{b+1-a}}{\left(  b+1-a\right)  \lambda^{b-a+1}}}\right)  \right\vert
^{2}e^{\frac{2|x|^{b+1-a}}{\left(  b+1-a\right)  \lambda^{b-a+1}}}dx
\end{align*}
and%
\[
\left(  \int_{\mathbb{R}^{N}}\dfrac{|\nabla u|^{2}}{|x|^{2b}}dx\right)
^{\frac{1}{2}}\left(  \int_{\mathbb{R}^{N}}\dfrac{|u|^{2}}{|x|^{2a}}dx\right)
^{\frac{1}{2}}\geq\left\vert \dfrac{N-3b+a-3}{2}\right\vert \left(
\int_{\mathbb{R}^{N}}\dfrac{|u|^{2}}{|x|^{a+b+1}}dx\right)  .
\]
This equality happens iff $u(x)=\alpha|x|^{2b+2-N}\exp\left(  \dfrac{\beta
}{b+1-a}|x|^{b+1-a}\right)  $ for some $\alpha\in\mathbb{R},\beta>0$.
\end{corollary}

\begin{proof}
Let $A=\frac{1}{|x|^{2b}}$ and $\overrightarrow{X}=\left(  |x|^{b-a}-\left(
N-2b-2\right)  \frac{1}{\left\vert x\right\vert }\right)  \frac{x}{|x|}$ in
Theorem \ref{T2.1}. Note that
\begin{align*}
-\operatorname{div}\left(  A\overrightarrow{X}\right)   &
=-\operatorname{div}\left(  |x|^{-a-b-1}x\right)  +\left(  N-2b-2\right)
\operatorname{div}\left(  |x|^{-2b-2}x\right) \\
&  =\frac{a+b+1-N}{|x|^{a+b+1}}+\frac{\left(  N-2b-2\right)  ^{2}}{|x|^{2b+2}}%
\end{align*}
and
\[
-\operatorname{div}\left(  A\overrightarrow{X}\right)  -A\left\vert
\overrightarrow{X}\right\vert ^{2}=\frac{N-3b+a-3}{|x|^{a+b+1}}-\frac
{1}{|x|^{2a}}.
\]
Now, we use Theorem \ref{T2.1}.
\end{proof}

\begin{corollary}
Let $b+1-a>0$ and $b\geq\dfrac{N-2}{2}$. Then for $u\in C_{0}^{\infty}\left(
\mathbb{R}^{N}\setminus\left\{  0\right\}  \right)  :$
\begin{align*}
&  \int_{\mathbb{R}^{N}}\dfrac{|\nabla u|^{2}}{|x|^{2b}}dx+\int_{\mathbb{R}%
^{N}}\dfrac{|u|^{2}}{|x|^{2a}}dx-\left(  3b-a+3-N\right)  \int_{\mathbb{R}%
^{N}}\dfrac{|u|^{2}}{|x|^{a+b+1}}dx\\
&  =\int_{\mathbb{R}^{N}}\dfrac{1}{|x|^{2N-2b-4}}\left\vert \nabla\left(
u|x|^{N-2b-2}e^{\frac{|x|^{b+1-a}}{b+1-a}}\right)  \right\vert ^{2}%
e^{-\frac{2|x|^{b+1-a}}{b+1-a}}dx.
\end{align*}
Also, for $u\in C_{0}^{\infty}\left(  \mathbb{R}^{N}\setminus\left\{
0\right\}  \right)  \setminus\left\{  0\right\}  $ and $\lambda=\left(
\frac{\int_{\mathbb{R}^{N}}\dfrac{|u|^{2}}{|x|^{2a}}dx}{\int_{\mathbb{R}^{N}%
}\dfrac{|\nabla u|^{2}}{|x|^{2b}}dx}\right)  ^{\frac{1}{2\left(  b+1-a\right)
}}:$%
\begin{align}
&  \left(  \int_{\mathbb{R}^{N}}\dfrac{|\nabla u|^{2}}{|x|^{2b}}dx\right)
^{\frac{1}{2}}\left(  \int_{\mathbb{R}^{N}}\dfrac{|u|^{2}}{|x|^{2a}}dx\right)
^{\frac{1}{2}}-\left\vert \dfrac{N-3b+a-3}{2}\right\vert \left(
\int_{\mathbb{R}^{N}}\dfrac{|u|^{2}}{|x|^{a+b+1}}dx\right) \nonumber\\
&  =\frac{1}{2}\lambda^{b-a+1}\int_{\mathbb{R}^{N}}\dfrac{1}{|x|^{2N-2b-4}%
}\left\vert \nabla\left(  u\left\vert x\right\vert ^{N-2b-2}e^{\frac
{|x|^{b+1-a}}{\left(  b+1-a\right)  \lambda^{b-a+1}}}\right)  \right\vert
^{2}e^{-\frac{2|x|^{b+1-a}}{\left(  b+1-a\right)  \lambda^{b-a+1}}}dx
\label{iCKN4}%
\end{align}
and%
\[
\left(  \int_{\mathbb{R}^{N}}\dfrac{|\nabla u|^{2}}{|x|^{2b}}dx\right)
^{\frac{1}{2}}\left(  \int_{\mathbb{R}^{N}}\dfrac{|u|^{2}}{|x|^{2a}}dx\right)
^{\frac{1}{2}}\geq\left\vert \dfrac{N-3b+a-3}{2}\right\vert \left(
\int_{\mathbb{R}^{N}}\dfrac{|u|^{2}}{|x|^{a+b+1}}dx\right)  .
\]
This equality happens iff $u(x)=\alpha|x|^{2b+2-N}\exp\left(  -\dfrac{\beta
}{b+1-a}|x|^{b+1-a}\right)  $ for some $\alpha\in\mathbb{R},\beta>0$.
\end{corollary}

\begin{proof}
Let $A=\frac{1}{|x|^{2b}}$ and $\overrightarrow{X}=\left(  -|x|^{b-a}-\left(
N-2b-2\right)  \frac{1}{\left\vert x\right\vert }\right)  \frac{x}{|x|}$ in
Theorem \ref{T2.1}.
\end{proof}

It is worth noting that in \cite{CFLL23}, the authors presented a simple
method to establish the stability of (\ref{2CKN1}). We now will provide an
equivalent approach to set up this stability result. We first recall a
weighted Poincar\'{e}\ inequality for the log-concave probability measure that
has been established in \cite{CFLL23}:

\begin{lemma}
For $\delta>0$, $N-2>\mu\geq0$ and $\alpha\geq\frac{N-2-\mu}{N-2}$:
\[
\int_{\mathbb{R}^{N}}\dfrac{|\nabla v(x)|^{2}}{|x|^{\mu}}e^{-\delta
|x|^{\alpha}}dx\geq C(N,\alpha,\delta,\mu)\inf_{c}\int_{\mathbb{R}^{N}}%
\dfrac{|v(x)-c|^{2}}{|x|^{\frac{N\mu}{N-2}}}e^{-\delta|x|^{\alpha}}dx.
\]

\end{lemma}

By making use of the scaling argument, we get the following estimate:

\begin{lemma}
\label{lem 2.7} For $\delta>0$, $\lambda>0$, $N-2>\mu\geq0$ and $\alpha
\geq\frac{N-2-\mu}{N-2}$:
\[
\lambda^{2+\mu-\frac{N\mu}{N-2}}\int_{\mathbb{R}^{N}}\dfrac{|\nabla v(x)|^{2}%
}{|x|^{\mu}}e^{-\delta\frac{|x|^{\alpha}}{\lambda^{\alpha}}}dx\geq
C(N,\alpha,\delta,\mu)\inf_{c}\int_{\mathbb{R}^{N}}\dfrac{|v(x)-c|^{2}%
}{|x|^{\frac{N\mu}{N-2}}}e^{-\delta\frac{|x|^{\alpha}}{\lambda^{\alpha}}}dx.
\]

\end{lemma}

By applying the above lemma and using the exact remainder term of
(\ref{2CKN1}), we obtain the stability for the $L^{2}$-CKN inequalities
(\ref{2CKN1}):

\begin{theorem}
Let $0\leq b<\frac{N-2}{2}$, $a<\frac{Nb}{N-2}$ and $a+b+1=\frac{2bN}{N-2}$.
There exists a universal constant $C(N,a,b)>0$ such that
\begin{align*}
&  \left(  \int_{\mathbb{R}^{N}}\dfrac{|\nabla u|^{2}}{|x|^{2b}}dx\right)
^{1/2}\left(  \int_{\mathbb{R}^{N}}\dfrac{|u|^{2}}{|x|^{2a}}dx\right)
^{1/2}-\frac{N-a-b-1}{2}\int_{\mathbb{R}^{N}}\dfrac{|u|^{2}}{|x|^{a+b+1}}dx\\
&  \geq C(N,a,b)\inf_{c\mathbb{\in R},\lambda>0}\int_{\mathbb{R}^{N}}%
\dfrac{\left\vert u-ce^{-\frac{\lambda}{b+1-a}|x|^{b+1-a}}\right\vert ^{2}%
}{|x|^{a+b+1}}dx.
\end{align*}

\end{theorem}

\begin{proof}
From \eqref{iCKN1} and Lemma \ref{lem 2.7}, with $\mu=2b$, $\delta=\frac
{2}{b+1-a}$, and $\alpha=b+1-a$, we get
\begin{align*}
&  \left(  \int_{\mathbb{R}^{N}}\dfrac{|\nabla u|^{2}}{|x|^{2b}}dx\right)
^{1/2}\left(  \int_{\mathbb{R}^{N}}\dfrac{|u|^{2}}{|x|^{2a}}dx\right)
^{1/2}-\frac{N-a-b-1}{2}\int_{\mathbb{R}^{N}}\dfrac{|u|^{2}}{|x|^{a+b+1}}dx\\
&  =\frac{\lambda^{b-a+1}}{2}\int_{\mathbb{R}^{N}}\dfrac{1}{|x|^{2b}%
}\left\vert \nabla\left(  ue^{\frac{|x|^{b+1-a}}{(b+1-a)\lambda^{b+1-a}}%
}\right)  \right\vert ^{2}e^{\frac{-2|x|^{b+1-a}}{(b+1-a)\lambda^{b+1-a}}}dx\\
&  \geq C(N,a,b)\inf_{c}\int_{\mathbb{R}^{N}}\dfrac{\left\vert ue^{\frac
{|x|^{b+1-a}}{(b+1-a)\lambda^{b+1-a}}}-c\right\vert ^{2}}{|x|^{\frac{2bN}%
{N-2}}}e^{\frac{-2|x|^{b+1-a}}{(b+1-a)\lambda^{b+1-a}}}dx\\
&  \geq C(N,a,b)\inf_{c\mathbb{\in R},\lambda>0}\int_{\mathbb{R}^{N}}%
\dfrac{\left\vert u-ce^{-\frac{|x|^{b+1-a}}{(b+1-a)\lambda^{b+1-a}}%
}\right\vert ^{2}}{|x|^{a+b+1}}dx.
\end{align*}

\end{proof}

Similarly, we can also establish the stability for the $L^{2}$-CKN
inequalities (\ref{2CKN4}):

\begin{theorem}
Let $\frac{N-2}{2}<b\leq N-2$ and $N\left(  b-a+3\right)  =2\left(
3b-a+3\right)  $. There exists a universal constant $C(N,a,b)>0$ such that
\begin{align*}
&  \left(  \int_{\mathbb{R}^{N}}\dfrac{|\nabla u|^{2}}{|x|^{2b}}dx\right)
^{1/2}\left(  \int_{\mathbb{R}^{N}}\dfrac{|u|^{2}}{|x|^{2a}}dx\right)
^{1/2}-\frac{3b-a-N+3}{2}\int_{\mathbb{R}^{N}}\dfrac{|u|^{2}}{|x|^{a+b+1}}dx\\
&  \geq C(N,a,b)\inf_{c\mathbb{\in R},\lambda>0}\int_{\mathbb{R}^{N}}%
\dfrac{\left\vert u-c|x|^{2b+2-N}e^{-\frac{\lambda}{b+1-a}|x|^{b+1-a}%
}\right\vert ^{2}}{|x|^{a+b+1}}dx.
\end{align*}

\end{theorem}

\begin{proof}
From \eqref{iCKN4} and Lemma \ref{lem 2.7}, with $\mu=2N-2b-4$, $\delta
=\frac{2}{b+1-a}$, and $\alpha=b+1-a$, we get
\begin{align*}
&  \left(  \int_{\mathbb{R}^{N}}\dfrac{|\nabla u|^{2}}{|x|^{2b}}dx\right)
^{1/2}\left(  \int_{\mathbb{R}^{N}}\dfrac{|u|^{2}}{|x|^{2a}}dx\right)
^{1/2}-\frac{3b-a-N+3}{2}\int_{\mathbb{R}^{N}}\dfrac{|u|^{2}}{|x|^{a+b+1}}dx\\
&  =\frac{\lambda^{b-a+1}}{2}\int_{\mathbb{R}^{N}}|x|^{4+2b-2N}\left\vert
\nabla\left(  u|x|^{N-2b-2}e^{\frac{|x|^{b+1-a}}{(b+1-a)\lambda^{b+1-a}}%
}\right)  \right\vert ^{2}e^{-\frac{2|x|^{b+1-a}}{(b+1-a)\lambda^{b+1-a}}}dx\\
&  \geq C(N,a,b)\inf_{c}\int_{\mathbb{R}^{N}}\dfrac{\left\vert u|x|^{N-2b-2}%
e^{\frac{|x|^{b+1-a}}{(b+1-a)\lambda^{b+1-a}}}-c\right\vert ^{2}}%
{|x|^{\frac{(a-b+1)N}{2}}}e^{\frac{-2|x|^{b+1-a}}{(b+1-a)\lambda^{b+1-a}}}dx\\
&  \geq C(N,a,b)\inf_{c}\int_{\mathbb{R}^{N}}\dfrac{\left\vert u|x|^{N-2b-2}%
-ce^{\frac{-|x|^{b+1-a}}{(b+1-a)\lambda^{b+1-a}}}\right\vert ^{2}}%
{|x|^{\frac{(a-b+1)N}{2}}}dx.
\end{align*}

\end{proof}

\section{$L^{p}$-Caffarelli-Kohn-Nirenberg inequalities and their stabilities}

We begin with the following $L^{p}$-Hardy and $L^{p}$%
-Caffarelli-Kohn-Nirenberg identities and their applications to get $L^{p}%
$-Hardy and $L^{p}$-Caffarelli-Kohn-Nirenberg inequalities. We note that the
weights in the following results are not radial.

\begin{theorem}
Let $N\geq1$, $p>1$, $0<R\leq\infty$, $V\geq0$ and $W$ be smooth functions on
$\left(  0,R\right)  $. If $\left(  r^{N+\left\vert P\right\vert
-1}V,r^{N+\left\vert P\right\vert -1}W\right)  $ is a $p$-Bessel pair on
$\left(  0,R\right)  $, that is, the ODE $\left(  r^{N+\left\vert P\right\vert
-1}V\left(  r\right)  \left\vert y^{\prime}\right\vert ^{p-2}y^{\prime
}\right)  ^{\prime}+r^{N+\left\vert P\right\vert -1}W\left(  r\right)
\left\vert y\right\vert ^{p-2}y=0$ has a positive solution $\varphi$ on
$\left(  0,R\right)  $, then for all $u\in C_{0}^{\infty}(B_{R}^{\ast
}\setminus\{0\}):$%
\[
{\int\limits_{B_{R}^{\ast}}}V\left(  \left\vert x\right\vert \right)
\left\vert \nabla u\right\vert ^{p}x^{P}\mathrm{dx}={\int\limits_{B_{R}^{\ast
}}}W\left(  \left\vert x\right\vert \right)  \left\vert u\right\vert ^{p}%
x^{P}dx+{\int\limits_{B_{R}^{\ast}}}V\left(  \left\vert x\right\vert \right)
\mathcal{R}_{p}\left(  u\frac{\varphi^{\prime}}{\varphi}\frac{x}{\left\vert
x\right\vert },\nabla u\right)  x^{P}dx,
\]%
\[
{\int\limits_{B_{R}^{\ast}}}V\left(  \left\vert x\right\vert \right)
\left\vert \frac{x}{\left\vert x\right\vert }\cdot\nabla u\right\vert
^{p}x^{P}\mathrm{dx}={\int\limits_{B_{R}^{\ast}}}W\left(  \left\vert
x\right\vert \right)  \left\vert u\right\vert ^{p}x^{P}dx+{\int\limits_{B_{R}%
^{\ast}}}V\left(  \left\vert x\right\vert \right)  \mathcal{R}_{p}\left(
u\left\vert \frac{\varphi^{\prime}}{\varphi}\right\vert ,\frac{x}{\left\vert
x\right\vert }\cdot\nabla u\right)  x^{P}dx,
\]%
\begin{align*}
&  \left(  {\int\limits_{B_{R}^{\ast}}}V\left(  \left\vert x\right\vert
\right)  \left\vert \nabla u\right\vert ^{p}x^{P}\mathrm{dx}\right)
^{\frac{1}{p}}\left(  {\int\limits_{B_{R}^{\ast}}}\left\vert \frac
{\varphi^{\prime}}{\varphi}\right\vert ^{p}V\left(  \left\vert x\right\vert
\right)  \left\vert u\right\vert ^{p}x^{P}\mathrm{dx}\right)  ^{\frac{p-1}{p}%
}\\
&  =\frac{1}{p}{\int\limits_{B_{R}^{\ast}}}\left[  W\left(  \left\vert
x\right\vert \right)  +\left(  p-1\right)  \left\vert \frac{\varphi^{\prime}%
}{\varphi}\right\vert ^{p}V\left(  \left\vert x\right\vert \right)  \right]
\left\vert u\right\vert ^{p}x^{P}\mathrm{dx}\\
&  +\frac{1}{p}{\int\limits_{B_{R}^{\ast}}}V\left(  \left\vert x\right\vert
\right)  \mathcal{R}_{p}\left(  \frac{\left\Vert V^{\frac{1}{p}}\nabla
ux^{\frac{P}{p}}\right\Vert _{p}^{\frac{1}{p}}}{\left\Vert \frac
{\varphi^{\prime}}{\varphi}V^{\frac{1}{p}}ux^{\frac{P}{p}}\right\Vert
_{p}^{\frac{1}{p}}}u\frac{\varphi^{\prime}}{\varphi}\frac{x}{\left\vert
x\right\vert },\frac{\left\Vert \frac{\varphi^{\prime}}{\varphi}V^{\frac{1}%
{p}}ux^{\frac{P}{p}}\right\Vert _{p}^{\frac{p-1}{p}}}{\left\Vert V^{\frac
{1}{p}}\nabla ux^{\frac{P}{p}}\right\Vert _{p}^{\frac{p-1}{p}}}\nabla
u\right)  x^{P}dx,
\end{align*}
and%
\begin{align*}
&  \left(  {\int\limits_{B_{R}^{\ast}}}V\left(  \left\vert x\right\vert
\right)  \left\vert \frac{x}{\left\vert x\right\vert }\cdot\nabla u\right\vert
^{p}x^{P}\mathrm{dx}\right)  ^{\frac{1}{p}}\left(  {\int\limits_{B_{R}^{\ast}%
}}\left\vert \frac{\varphi^{\prime}}{\varphi}\right\vert ^{p}V\left(
\left\vert x\right\vert \right)  \left\vert u\right\vert ^{p}x^{P}%
\mathrm{dx}\right)  ^{\frac{p-1}{p}}\\
&  =\frac{1}{p}{\int\limits_{B_{R}^{\ast}}}\left[  W\left(  \left\vert
x\right\vert \right)  +\left(  p-1\right)  \left\vert \frac{\varphi^{\prime}%
}{\varphi}\right\vert ^{p}V\left(  \left\vert x\right\vert \right)  \right]
\left\vert u\right\vert ^{p}x^{P}\mathrm{dx}\\
&  +\frac{1}{p}{\int\limits_{B_{R}^{\ast}}}V\left(  \left\vert x\right\vert
\right)  \mathcal{R}_{p}\left(  \frac{\left\Vert V^{\frac{1}{p}}\frac
{x}{\left\vert x\right\vert }\cdot\nabla ux^{\frac{P}{p}}\right\Vert
_{p}^{\frac{1}{p}}}{\left\Vert \frac{\varphi^{\prime}}{\varphi}V^{\frac{1}{p}%
}ux^{\frac{P}{p}}\right\Vert _{p}^{\frac{1}{p}}}u\left\vert \frac
{\varphi^{\prime}}{\varphi}\right\vert ,\frac{\left\Vert \frac{\varphi
^{\prime}}{\varphi}V^{\frac{1}{p}}ux^{\frac{P}{p}}\right\Vert _{p}^{\frac
{p-1}{p}}}{\left\Vert V^{\frac{1}{p}}\frac{x}{\left\vert x\right\vert }%
\cdot\nabla ux^{\frac{P}{p}}\right\Vert _{p}^{\frac{p-1}{p}}}\frac
{x}{\left\vert x\right\vert }\cdot\nabla u\right)  x^{P}dx.
\end{align*}
Therefore%
\[
{\int\limits_{B_{R}^{\ast}}}V\left(  \left\vert x\right\vert \right)
\left\vert \nabla u\right\vert ^{p}x^{P}\mathrm{dx}\geq{\int\limits_{B_{R}%
^{\ast}}}V\left(  \left\vert x\right\vert \right)  \left\vert \frac
{x}{\left\vert x\right\vert }\cdot\nabla u\right\vert ^{p}x^{P}\mathrm{dx}%
\geq{\int\limits_{B_{R}^{\ast}}}W\left(  \left\vert x\right\vert \right)
\left\vert u\right\vert ^{p}x^{P}\mathrm{dx}%
\]
and%
\begin{align*}
&  \left(  {\int\limits_{B_{R}^{\ast}}}V\left(  \left\vert x\right\vert
\right)  \left\vert \nabla u\right\vert ^{p}x^{P}\mathrm{dx}\right)
^{\frac{1}{p}}\left(  {\int\limits_{B_{R}^{\ast}}}\left\vert \frac
{\varphi^{\prime}}{\varphi}\right\vert ^{p}V\left(  \left\vert x\right\vert
\right)  \left\vert u\right\vert ^{p}x^{P}\mathrm{dx}\right)  ^{\frac{p-1}{p}%
}\\
&  \geq\left(  {\int\limits_{B_{R}^{\ast}}}V\left(  \left\vert x\right\vert
\right)  \left\vert \frac{x}{\left\vert x\right\vert }\cdot\nabla u\right\vert
^{p}x^{P}\mathrm{dx}\right)  ^{\frac{1}{p}}\left(  {\int\limits_{B_{R}^{\ast}%
}}\left\vert \frac{\varphi^{\prime}}{\varphi}\right\vert ^{p}V\left(
\left\vert x\right\vert \right)  \left\vert u\right\vert ^{p}x^{P}%
\mathrm{dx}\right)  ^{\frac{p-1}{p}}\\
&  \geq\frac{1}{p}{\int\limits_{B_{R}^{\ast}}}\left[  W\left(  \left\vert
x\right\vert \right)  +\left(  p-1\right)  \left\vert \frac{\varphi^{\prime}%
}{\varphi}\right\vert ^{p}V\left(  \left\vert x\right\vert \right)  \right]
\left\vert u\right\vert ^{p}x^{P}\mathrm{dx}\text{.}%
\end{align*}
Here $x^{P}=\left\vert x_{1}\right\vert ^{P_{1}}...\left\vert x_{N}\right\vert
^{P_{N}}$, $P_{1}\geq0,...,$ $P_{N}\geq0$, is the monomial weight, $\left\vert
P\right\vert =P_{1}+...+P_{N}$, $\mathbb{R}_{\ast}^{N}=\left\{  \left(
x_{1},...,x_{N}\right)  \in\mathbb{R}^{N}:x_{i}>0\text{ whenever }%
P_{i}>0\right\}  $, and $B_{R}^{\ast}=B_{R}\cap\mathbb{R}_{\ast}^{N}$.
\end{theorem}

\begin{proof}
Choose $A=Vx^{P}$ and $\overrightarrow{X}=\frac{\varphi^{\prime}\left(
\left\vert x\right\vert \right)  }{\varphi\left(  \left\vert x\right\vert
\right)  }\frac{x}{\left\vert x\right\vert }$. Then%
\begin{align*}
\operatorname{div}\left(  A\left\vert F\right\vert ^{p-2}\overrightarrow
{X}\right)   &  =\operatorname{div}\left(  V\left\vert \frac{\varphi^{\prime}%
}{\varphi}\right\vert ^{p-2}\frac{\varphi^{\prime}}{\varphi}\frac
{x}{\left\vert x\right\vert }x^{P}\right) \\
&  =\nabla\left(  V\left\vert \frac{\varphi^{\prime}}{\varphi}\right\vert
^{p-2}\frac{\varphi^{\prime}}{\varphi}\right)  \cdot\frac{x}{\left\vert
x\right\vert }x^{P}+V\left\vert \frac{\varphi^{\prime}}{\varphi}\right\vert
^{p-2}\frac{\varphi^{\prime}}{\varphi}\operatorname{div}\left(  \frac
{x}{\left\vert x\right\vert }x^{P}\right) \\
&  =\frac{\left(  V\left\vert \varphi^{\prime}\right\vert ^{p-2}%
\varphi^{\prime}\right)  ^{\prime}}{\varphi^{p-1}}-\left(  p-1\right)
\frac{V\left\vert \varphi^{\prime}\right\vert ^{p}}{\varphi^{p}}%
+\frac{N+\left\vert P\right\vert -1}{\left\vert x\right\vert }V\frac
{\left\vert \varphi^{\prime}\right\vert ^{p-2}\varphi^{\prime}}{\varphi^{p-1}%
}\\
&  =-W-\left(  p-1\right)  \frac{V\left\vert \varphi^{\prime}\right\vert ^{p}%
}{\varphi^{p}}.
\end{align*}
Now we apply Theorem \ref{T2} and Theorem \ref{T3}.
\end{proof}

We can also derive the following $L^{p}$-CKN inequalities with exact remainder terms:

\begin{theorem}
Let $N\geq1,~p>1$, $b+1-a>0$ and $b\leq\frac{N-p}{p}$. For any $u\in
C_{0}^{\infty}(\mathbb{R}^{N}\setminus\{0\})$, there hold
\begin{align*}
&  \int_{\mathbb{R}^{N}}\dfrac{|\nabla u|^{p}}{|x|^{pb}}dx+(p-1)\int
_{\mathbb{R}^{N}}\dfrac{|u|^{p}}{|x|^{pa}}dx-(N-1-(p-1)a-b)\int_{\mathbb{R}%
^{N}}\dfrac{|u|^{p}}{|x|^{(p-1)a+b+1}}dx\\
&  =\int_{\mathbb{R}^{N}}\dfrac{1}{|x|^{pb}}\mathcal{R}_{p}\left(
-u|x|^{b-1-a}x,\nabla u\right)  dx,
\end{align*}%
\begin{align*}
&  \int_{\mathbb{R}^{N}}\dfrac{\left\vert \frac{x}{\left\vert x\right\vert
}\cdot\nabla u\right\vert ^{p}}{|x|^{pb}}dx+(p-1)\int_{\mathbb{R}^{N}}%
\dfrac{|u|^{p}}{|x|^{pa}}dx-(N-1-(p-1)a-b)\int_{\mathbb{R}^{N}}\dfrac{|u|^{p}%
}{|x|^{(p-1)a+b+1}}dx\\
&  =\int_{\mathbb{R}^{N}}\dfrac{1}{|x|^{pb}}\mathcal{R}_{p}\left(
u|x|^{b-a},-\frac{x}{\left\vert x\right\vert }\cdot\nabla u\right)  dx,
\end{align*}%
\begin{align}
&  \left(  \int_{\mathbb{R}^{N}}\dfrac{|\nabla u|^{p}}{|x|^{pb}}dx\right)
^{\frac{1}{p}}\left(  \int_{\mathbb{R}^{N}}\dfrac{|u|^{p}}{|x|^{pa}}dx\right)
^{\frac{p-1}{p}}-\frac{N-1-\left(  p-1\right)  a-b}{p}\int_{\mathbb{R}^{N}%
}\dfrac{|u|^{p}}{|x|^{(p-1)a+b+1}}dx\nonumber\\
&  =\frac{1}{p}\int_{\Omega}\dfrac{1}{|x|^{pb}}\mathcal{R}_{p}\left(  -\left(
\frac{\int_{\mathbb{R}^{N}}\dfrac{|\nabla u|^{p}}{|x|^{pb}}dx}{\int
_{\mathbb{R}^{N}}\dfrac{|u|^{p}}{|x|^{pa}}dx}\right)  ^{\frac{1}{p^{2}}%
}u|x|^{b-1-a}x,\left(  \frac{\int_{\mathbb{R}^{N}}\dfrac{|u|^{p}}{|x|^{pa}}%
dx}{\int_{\mathbb{R}^{N}}\dfrac{|\nabla u|^{p}}{|x|^{pb}}dx}\right)
^{\frac{p-1}{p^{2}}}\nabla u\right)  dx, \label{pCKN}%
\end{align}
and%
\begin{align*}
&  \left(  \int_{\mathbb{R}^{N}}\dfrac{\left\vert \frac{x}{\left\vert
x\right\vert }\cdot\nabla u\right\vert ^{p}}{|x|^{pb}}dx\right)  ^{\frac{1}%
{p}}\left(  \int_{\mathbb{R}^{N}}\dfrac{|u|^{p}}{|x|^{pa}}dx\right)
^{\frac{p-1}{p}}-\frac{N-1-\left(  p-1\right)  a-b}{p}\int_{\mathbb{R}^{N}%
}\dfrac{|u|^{p}}{|x|^{(p-1)a+b+1}}dx\\
&  =\frac{1}{p}\int_{\Omega}\dfrac{1}{|x|^{pb}}\mathcal{R}_{p}\left(  \left(
\frac{\int_{\mathbb{R}^{N}}\dfrac{\left\vert \frac{x}{\left\vert x\right\vert
}\cdot\nabla u\right\vert ^{p}}{|x|^{pb}}dx}{\int_{\mathbb{R}^{N}}%
\dfrac{|u|^{p}}{|x|^{pa}}dx}\right)  ^{\frac{1}{p^{2}}}u|x|^{b-a},-\left(
\frac{\int_{\mathbb{R}^{N}}\dfrac{|u|^{p}}{|x|^{pa}}dx}{\int_{\mathbb{R}^{N}%
}\dfrac{\left\vert \frac{x}{\left\vert x\right\vert }\cdot\nabla u\right\vert
^{p}}{|x|^{pb}}dx}\right)  ^{\frac{p-1}{p^{2}}}\frac{x}{\left\vert
x\right\vert }\cdot\nabla u\right)  dx.
\end{align*}

\begin{proof}
Apply Theorem \ref{T2} and Theorem \ref{T3} with $A=\dfrac{1}{|x|^{pb}}$ and
$\overrightarrow{X}=-|x|^{b-1-a}x$.
\end{proof}
\end{theorem}

\begin{theorem}
Let $N\geq1,~p>1$, $b+1-a<0$ and $b\geq\frac{N-p}{p}$. For any $u\in
C_{0}^{\infty}(\mathbb{R}^{N}\setminus\{0\})$, there hold
\begin{align*}
&  \int_{\mathbb{R}^{N}}\dfrac{|\nabla u|^{p}}{|x|^{pb}}dx+(p-1)\int
_{\mathbb{R}^{N}}\dfrac{|u|^{p}}{|x|^{pa}}dx-(1+(p-1)a+b-N)\int_{\mathbb{R}%
^{N}}\dfrac{|u|^{p}}{|x|^{(p-1)a+b+1}}dx\\
&  =\int_{\mathbb{R}^{N}}\dfrac{1}{|x|^{pb}}\mathcal{R}_{p}\left(
u|x|^{b-1-a}x,\nabla u\right)  dx,
\end{align*}%
\begin{align*}
&  \int_{\mathbb{R}^{N}}\dfrac{\left\vert \frac{x}{\left\vert x\right\vert
}\cdot\nabla u\right\vert ^{p}}{|x|^{pb}}dx+(p-1)\int_{\mathbb{R}^{N}}%
\dfrac{|u|^{p}}{|x|^{pa}}dx-(1+(p-1)a+b-N)\int_{\mathbb{R}^{N}}\dfrac{|u|^{p}%
}{|x|^{(p-1)a+b+1}}dx\\
&  =\int_{\mathbb{R}^{N}}\dfrac{1}{|x|^{pb}}\mathcal{R}_{p}\left(
u|x|^{b-a},\frac{x}{\left\vert x\right\vert }\cdot\nabla u\right)  dx,
\end{align*}%
\begin{align*}
&  \left(  \int_{\mathbb{R}^{N}}\dfrac{|\nabla u|^{p}}{|x|^{pb}}dx\right)
^{\frac{1}{p}}\left(  \int_{\mathbb{R}^{N}}\dfrac{|u|^{p}}{|x|^{pa}}dx\right)
^{\frac{p-1}{p}}-\frac{1+(p-1)a+b-N}{p}\int_{\mathbb{R}^{N}}\dfrac{|u|^{p}%
}{|x|^{(p-1)a+b+1}}dx\\
&  =\frac{1}{p}\int_{\Omega}\dfrac{1}{|x|^{pb}}\mathcal{R}_{p}\left(  \left(
\frac{\int_{\mathbb{R}^{N}}\dfrac{|\nabla u|^{p}}{|x|^{pb}}dx}{\int
_{\mathbb{R}^{N}}\dfrac{|u|^{p}}{|x|^{pa}}dx}\right)  ^{\frac{1}{p^{2}}%
}u|x|^{b-1-a}x,\left(  \frac{\int_{\mathbb{R}^{N}}\dfrac{|u|^{p}}{|x|^{pa}}%
dx}{\int_{\mathbb{R}^{N}}\dfrac{|\nabla u|^{p}}{|x|^{pb}}dx}\right)
^{\frac{p-1}{p^{2}}}\nabla u\right)  dx,
\end{align*}
and%
\begin{align*}
&  \left(  \int_{\mathbb{R}^{N}}\dfrac{\left\vert \frac{x}{\left\vert
x\right\vert }\cdot\nabla u\right\vert ^{p}}{|x|^{pb}}dx\right)  ^{\frac{1}%
{p}}\left(  \int_{\mathbb{R}^{N}}\dfrac{|u|^{p}}{|x|^{pa}}dx\right)
^{\frac{p-1}{p}}-\frac{1+(p-1)a+b-N}{p}\int_{\mathbb{R}^{N}}\dfrac{|u|^{p}%
}{|x|^{(p-1)a+b+1}}dx\\
&  =\frac{1}{p}\int_{\Omega}\dfrac{1}{|x|^{pb}}\mathcal{R}_{p}\left(  \left(
\frac{\int_{\mathbb{R}^{N}}\dfrac{\left\vert \frac{x}{\left\vert x\right\vert
}\cdot\nabla u\right\vert ^{p}}{|x|^{pb}}dx}{\int_{\mathbb{R}^{N}}%
\dfrac{|u|^{p}}{|x|^{pa}}dx}\right)  ^{\frac{1}{p^{2}}}u|x|^{b-a},\left(
\frac{\int_{\mathbb{R}^{N}}\dfrac{|u|^{p}}{|x|^{pa}}dx}{\int_{\mathbb{R}^{N}%
}\dfrac{\left\vert \frac{x}{\left\vert x\right\vert }\cdot\nabla u\right\vert
^{p}}{|x|^{pb}}dx}\right)  ^{\frac{p-1}{p^{2}}}\frac{x}{\left\vert
x\right\vert }\cdot\nabla u\right)  dx.
\end{align*}

\begin{proof}
Apply Theorem \ref{T2} and Theorem \ref{T3} with $A=\dfrac{1}{|x|^{pb}}$ and
$\overrightarrow{X}=|x|^{b-1-a}x$.
\end{proof}
\end{theorem}

Using the information on the remainder terms, we obtain the following $L^{p}%
$-CKN inequalities with sharp constants and explicit optimizers:

\begin{theorem}
Let $N\geq1,~p>1$. Then for any $u\in C_{0}^{\infty}(\mathbb{R}^{N}%
\setminus\{0\})$, there holds
\[
\left(  \int_{\mathbb{R}^{N}}\dfrac{|\nabla u|^{p}}{|x|^{pb}}dx\right)
^{\frac{1}{p}}\left(  \int_{\mathbb{R}^{N}}\dfrac{|u|^{p}}{|x|^{pa}}dx\right)
^{\frac{p-1}{p}}\geq\frac{\left\vert N-1-\left(  p-1\right)  a-b\right\vert
}{p}\int_{\mathbb{R}^{N}}\dfrac{|u|^{p}}{|x|^{(p-1)a+b+1}}dx.
\]
Also,

\begin{enumerate}
\item If $b+1-a>0$ and $b\leq\frac{N-p}{p}$, then the constant $\frac
{N-1-\left(  p-1\right)  a-b}{p}$ is sharp and can be attained only by the
functions of the form $u(x)=D\exp(\frac{t|x|^{b+1-a}}{b+1-a}),$ $t<0$.

\item If $b+1-a<0$ and $b\geq\frac{N-p}{p}$, then the constant $\frac
{1+(p-1)a+b-N}{p}$ is sharp and can be attained only by the functions of the
form $u(x)=D\exp(\frac{t|x|^{b+1-a}}{b+1-a}),$ $t>0$.
\end{enumerate}
\end{theorem}

Our next goal is to study the stability of the above $L^{p}$-CKN inequalities.
We will follow the approach in \cite{CFLL23}. In order to do that, we will
first establish a weighted $L^{p}$-Poincar\'{e}\ inequality for the
log-concave probability measure which is of independent interest.

\begin{lemma}
For some $\delta>0$, $N-p>\mu\geq0$ and $\alpha\geq\frac{N-p-\mu}{N-p}$, we
have for $v\in C_{0}^{\infty}(\mathbb{R}^{N}\setminus\{0\})$ that
\[
\int_{\mathbb{R}^{N}}\dfrac{|\nabla v(y)|^{p}}{|y|^{\mu}}e^{-\delta
|y|^{\alpha}}dy\geq C(N,p,\alpha,\delta,\mu)\inf_{c}\int_{\mathbb{R}^{N}%
}\dfrac{|v(y)-c|^{p}}{|y|^{\frac{N\mu}{N-p}}}e^{-\delta|y|^{\alpha}}dy.
\]

\end{lemma}

\begin{proof}
Let $\overline{v}(x)=\left(  \frac{1}{\lambda}\right)  ^{\frac{1}{p^{\ast}}%
}v\left(  |x|^{\lambda-1}x\right)  $. From \cite{LL17}, we have the Jacobian
for the change of variable $x\rightarrow|x|^{\lambda-1}x$ is $\lambda
|x|^{N(\lambda-1)}$, and for $\lambda\geq1$, we can get the following
estimate
\[
|\nabla\overline{v}(x)|\leq\lambda^{\frac{1}{p}}|x|^{\lambda-1}\left\vert
\nabla v\left(  |x|^{\lambda-1}x\right)  \right\vert .
\]
By setting $y=|x|^{\lambda-1}x$, we obtain
\begin{align*}
\int_{\mathbb{R}^{N}}\dfrac{|\nabla v(y)|^{p}}{|y|^{\mu}}e^{-\delta
|y|^{\alpha}}dy  &  =\int_{\mathbb{R}^{N}}\dfrac{\left\vert \nabla v\left(
|x|^{\lambda-1}x\right)  \right\vert ^{p}}{|x|^{\lambda\mu}}e^{-\delta
|x|^{\lambda a}}\lambda|x|^{N(\lambda-1)}dx\\
&  \geq\int_{\mathbb{R}^{N}}\dfrac{\left\vert \nabla\overline{v}(x)\right\vert
^{p}}{\lambda|x|^{p(\lambda-1)+\lambda\mu-N(\lambda-1)}}e^{-\delta|x|^{\lambda
a}}\lambda dx\\
&  =\int_{\mathbb{R}^{N}}\dfrac{\left\vert \nabla\overline{v}(x)\right\vert
^{p}}{|x|^{\lambda(p+\mu-N)+N-p}}e^{-\delta|x|^{\lambda a}}dx.
\end{align*}
Choosing $\lambda=\frac{N-p}{N-p-\mu}\geq1$, making use of Theorem 2.4 in
\cite{Mil09} with the fact that $e^{-\delta|x|^{\frac{N-p}{N-p-\mu}\alpha}}dx$
is a log-concave measure for all $\alpha\geq\frac{N-p-\mu}{N-p}$, we get
\begin{align*}
\int_{\mathbb{R}^{N}}\dfrac{|\nabla v(y)|^{p}}{|y|^{\mu}}e^{-\delta
|y|^{\alpha}}dy  &  \geq\int_{\mathbb{R}^{N}}\left\vert \nabla\overline
{v}(x)\right\vert ^{p}e^{-\delta|x|^{\frac{N-p}{N-p-\mu}\alpha}}dx\\
&  \geq C_{1}(N,p,\alpha,\delta,\mu)\inf_{c}\int_{\mathbb{R}^{N}}|\overline
{v}(x)-c|^{p}e^{-\delta|x|^{\frac{N-p}{N-p-\mu}\alpha}}dx\\
&  =C_{1}(N,p,\alpha,\delta,\mu)\inf_{c}\int_{\mathbb{R}^{N}}\left\vert
\left(  \frac{1}{\lambda}\right)  ^{\frac{1}{p^{\ast}}}v\left(  |x|^{\lambda
-1}x\right)  -c\right\vert ^{p}e^{-\delta|x|^{\frac{N-p}{N-p-\mu}\alpha}}dx\\
&  =C_{2}(N,p,\alpha,\delta,\mu)\inf_{c}\int_{\mathbb{R}^{N}}\dfrac{\left\vert
v\left(  |x|^{\lambda-1}x\right)  -c\right\vert ^{p}}{\lambda|x|^{N(\lambda
-1)}}e^{-\delta|x|^{\lambda\alpha}\lambda|x|^{N(\lambda-1)}}dx\\
&  =C(N,p,\alpha,\delta,\mu)\inf_{c}\int_{\mathbb{R}^{N}}\dfrac{\left\vert
v\left(  y\right)  -c\right\vert ^{p}}{|y|^{\frac{N\mu}{N-p}}}e^{-\delta
|y|^{\alpha}}dy.
\end{align*}

\end{proof}

By the scaling argument, we obtain the following weighted $L^{p}$-Poincar\'{e}\ inequality:

\begin{corollary}
\label{cor 4.2} For some $\delta>0$, $N-p>\mu\geq0$, $\alpha\geq\frac{N-p-\mu
}{N-p}$ and $\lambda>0$, we have for $v\in C_{0}^{\infty}(\mathbb{R}%
^{N}\setminus\{0\})$ that
\[
\lambda^{p+\mu-\frac{N\mu}{N-p}}\int_{\mathbb{R}^{N}}\dfrac{|\nabla v(y)|^{p}%
}{|y|^{\mu}}e^{-\delta\frac{|y|^{\alpha}}{\lambda^{\alpha}}}dy\geq
C(N,p,\alpha,\delta,\mu)\inf_{c}\int_{\mathbb{R}^{N}}\dfrac{|v(y)-c|^{p}%
}{|y|^{\frac{N\mu}{N-p}}}e^{-\delta\frac{|y|^{\alpha}}{\lambda^{\alpha}}}dy.
\]

\end{corollary}

Now, we will apply the above lemma to get a result about the $L^{p}$-stability
for the CKN inequalites, i.e.

\begin{theorem}
Let $p\geq2$, $0\leq b<\frac{N-p}{p}$, $a<\frac{Nb}{N-p}$ and
$(p-1)a+b+1=\frac{pbN}{N-p}$. There exists a universal constant $C(N,p,a,b)>0$
such that for all $u\in C_{0}^{\infty}(\mathbb{R}^{N}\setminus\{0\}):$
\begin{align*}
&  \left(  \int_{\mathbb{R}^{N}}\dfrac{|\nabla u|^{p}}{|x|^{pb}}dx\right)
^{\frac{1}{p}}\left(  \int_{\mathbb{R}^{N}}\dfrac{|u|^{p}}{|x|^{pa}}dx\right)
^{\frac{p-1}{p}}-\dfrac{N-1-\left(  p-1\right)  a-b}{p}\int_{\mathbb{R}^{N}%
}\dfrac{|u|^{p}}{|x|^{(p-1)a+b+1}}dx\\
&  \geq C(N,p,a,b)\inf_{c\mathbb{\in R},\lambda>0}\int_{\mathbb{R}^{N}}%
\dfrac{\left\vert u-ce^{-\frac{\lambda}{b+1-a}|x|^{b+1-a}}\right\vert ^{p}%
}{|x|^{(p-1)a+b+1}}dx.
\end{align*}

\end{theorem}

\begin{proof}
From \eqref{pCKN} and Lemma \ref{L1}, we get with $\lambda=\left(  \dfrac
{\int_{\mathbb{R}^{N}}|u|^{p}/|x|^{pa}dx}{\int_{\mathbb{R}^{N}}|\nabla
u|^{p}/|x|^{pb}dx}\right)  ^{1/(p(b+1-a))}$ that
\begin{align}
&  \left(  \int_{\mathbb{R}^{N}}\dfrac{|\nabla u|^{p}}{|x|^{pb}}\right)
^{\frac{1}{p}}\left(  \int_{\mathbb{R}^{N}}\dfrac{|u|^{p}}{|x|^{pa}}\right)
^{\frac{p-1}{p}}-\dfrac{N-1+(1-p)a-b}{p}\int_{\mathbb{R}^{N}}\dfrac{|u|^{p}%
}{|x|^{(p-1)a+b+1}}dx\nonumber\label{ipCKN}\\
&  \geq\frac{M_{p}}{p}\lambda^{(p-1)(b+1-a)}\int_{\mathbb{R}^{N}}\dfrac
{1}{|x|^{pb}}\left\vert \nabla\left(  ue^{\frac{|x|^{b+1-a}}{(b+1-a)\lambda
^{b+1-a}}}\right)  \right\vert ^{p}e^{-\frac{p|x|^{b+1-a}}{(b+1-a)\lambda
^{b+1-a}}}dx.
\end{align}

Now, apply Corollary \ref{cor 4.2}, with $\mu=pb$, $\delta=\frac{p}{b+1-a}$,
and $\alpha=b+1-a$, we get
\begin{align*}
&  \left(  \int_{\mathbb{R}^{N}}\dfrac{|\nabla u|^{p}}{|x|^{pb}}\right)
^{\frac{1}{p}}\left(  \int_{\mathbb{R}^{N}}\dfrac{|u|^{p}}{|x|^{pa}}\right)
^{\frac{p-1}{p}}-\dfrac{N-1-\left(  p-1\right)  a-b}{p}\int_{\mathbb{R}^{N}%
}\dfrac{|u|^{p}}{|x|^{(p-1)a+b+1}}dx\\
&  \geq\frac{M_{p}}{p}\lambda^{(p-1)(b+1-a)}\int_{\mathbb{R}^{N}}\dfrac
{1}{|x|^{pb}}\left\vert \nabla\left(  ue^{\frac{|x|^{b+1-a}}{(b+1-a)\lambda
^{b+1-a}}}\right)  \right\vert ^{p}e^{-\frac{p|x|^{b+1-a}}{(b+1-a)\lambda
^{b+1-a}}}dx\\
&  \geq C(N,p,a,b)\inf_{c}\int_{\mathbb{R}^{N}}\dfrac{\left\vert
ue^{\frac{|x|^{b+1-a}}{(b+1-a)\lambda^{b+1-a}}}-c\right\vert ^{p}}%
{|x|^{\frac{pbN}{N-p}}}e^{\frac{-p|x|^{b+1-a}}{(b+1-a)\lambda^{b+1-a}}}dx\\
&  \geq C(N,p,a,b)\inf_{c\mathbb{\in R},\lambda>0}\int_{\mathbb{R}^{N}}%
\dfrac{\left\vert u-ce^{-\frac{|x|^{b+1-a}}{(b+1-a)\lambda^{b+1-a}}%
}\right\vert ^{p}}{|x|^{(p-1)a+b+1}}dx.
\end{align*}

\end{proof}

\end{document}